\documentclass[reqno, a4paper, 11pt]{amsart}
\usepackage{graphicx} % Required for inserting images

\usepackage{amsfonts}
\usepackage[dvipsnames]{xcolor}
\usepackage{amssymb}
\usepackage{amsthm}
\usepackage{latexsym,amsmath}
\usepackage{afterpage}
\usepackage[T1]{fontenc}
\usepackage{epigraph}
\usepackage[shortlabels]{enumitem}
\usepackage{geometry}
\usepackage{orcidlink}
\usepackage{hyperref}
\hypersetup{colorlinks=true, citecolor=green, linkcolor=red}

\newcommand{\R}{\mathbb{R}}

\newcommand{\N}{\mathbb{N}}

\newcommand{\V}{\mathbb{V}}
\newcommand{\W}{\mathbb{W}}
\newcommand{\XB}{\mathbb{X}}
\newcommand{\YB}{\mathbb{Y}}
\newcommand{\PP}{\mathbb{P}}
\newcommand{\HB}{\mathbb{H}}
\newcommand{\LB}{\mathbb{L}}

\theoremstyle{plain}
\newtheorem{theorem}{Theorem}[section]
\newtheorem{lemma}[theorem]{Lemma}

\newtheorem{proposition}[theorem]{Proposition}

\theoremstyle{definition}
\newtheorem{definition}[theorem]{Definition}

\theoremstyle{remark}

\numberwithin{equation}{section}

\newcommand{\RR}{\mathcal{R}}
\newcommand{\lnorm}[2]{\left\|#1\right\|_{\LB^{#2}(\Omega)}}

\newcommand{\hnorm}[2]{\left\|#1\right\|_{\HB^{#2}(\Omega)}}

\newcommand{\wnorm}[3]{\left\|#1\right\|_{\W^{#2,#3}(\Omega)}}

\newcommand{\liprod}[2]{\left\langle#1,#2\right\rangle}

\newcommand{\intOt}[1]{\int_0^t #1 \; {\rm d}s}
\renewcommand{\vec}[1]{\mbox{\boldmath $ #1 $}}

\newcommand{\ddt}{\frac{\rm{d}}{{\rm d}t}}

\let\div\undefined
\DeclareMathOperator{\div}{div}
\DeclareMathOperator{\curl}{curl}

\newcommand{\pmat}[1]{\begin{pmatrix}#1\end{pmatrix}}

\newcommand{\weakto}{\rightharpoonup}
\newcommand{\wkstarto}{\overset{\star}{\rightharpoonup}}
\newcommand{\lsim}{\lesssim}

\newcounter{Ax}

\allowdisplaybreaks

\newcommand{\comm}[1]{\textcolor{blue}{#1}}

\usepackage[normalem]{ulem}
\newcommand\tsout{\bgroup\markoverwith{\textcolor{red}{\rule[0.5ex]{2pt}{1.4pt}}}\ULon}
\newcommand{\stkout}[1]{\ifmmode\text{\tsout{\ensuremath{#1}}}\else\tsout{#1}\fi}

\title[Well-posedness for a Magnetohydrodynamical Model with Intrinsic Magnetisation]{Well-posedness for a Magnetohydrodynamical Model with Intrinsic Magnetisation}
\author{{Noah Vinod\orcidlink{0009-0008-4388-2329} and Thanh Tran\orcidlink{0000-0001-6117-4811}}}
\date{\today}
\address{School of Mathematics and Statistics, The University of New South Wales, Sydney 2052, Australia}
\email{\comm{n.vinod@unsw.edu.au}}
\address{School of Mathematics and Statistics, The University of New South Wales, Sydney 2052, Australia}
\email{\comm{thanh.tran@unsw.edu.au}}

\begin{document}

\begin{abstract}
Ferromagnetic magnetohydrodynamics concerns the study of conducting fluids with intrinsic magnetisation under the influence of a magnetic field. It is a generalisation of the magnetohydrodynamical equations and takes into account the dynamics of the magnetisation of a fluid. First proposed by Lingam (Lingam, `Dissipative effects in magnetohydrodynamical models with intrinsic magnetisation', Communications in Nonlinear Science and Numerical Simulation Vol 28, pp 223-231, 2015), the usual equations of magnetohydrodynamics, namely the Navier-Stokes equation and the induction equation, are coupled with the Landau-Lifshitz-Gilbert equation. In this paper, the local existence, uniqueness and regularity of weak solutions to this system are discussed.
\end{abstract}

\maketitle

\tableofcontents

\section{Introduction} \label{sec:introduction}

% potential MHD math papers
% https://link.springer.com/article/10.1007/BF02355589
% https://projecteuclid.org/journals/tohoku-mathematical-journal/volume-41/issue-3/Weak-and-classical-solutions-of-the-two-dimensional-magnetohydrodynamic-equations/10.2748/tmj/1178227774.full
% https://projecteuclid.org/journals/proceedings-of-the-japan-academy-series-a-mathematical-sciences/volume-64/issue-6/Initial-boundary-value-problem-for-the-equations-of-ideal-magneto/10.3792/pjaa.64.191.full
% https://projecteuclid.org/journals/hokkaido-mathematical-journal/volume-16/issue-3/The-initial-boundary-value-problem-for-the-equations-of-ideal/10.14492/hokmj/1381518181.full
% https://www.mathnet.ru/php/archive.phtml?wshow=paper&jrnid=tm&paperid=1438&option_lang=eng

In this paper, we develop local-in-time existence and uniqueness theorems the ferromagnetic magnetohydrodynamical model represented by the following system of PDEs 
\begin{subequations} \label{eqn:fmhd}
    \begin{alignat}{2}
    &\partial_t \vec{v} + (\vec{v} \cdot \nabla) \vec{v} - \mu\Delta \vec{v} + \nabla p = \curl \vec{B} \times \vec{B} + (\vec{B} \cdot \nabla) \vec{m} + \nabla (\vec{m} \cdot \vec{B}) - \div \left[(\nabla \vec{m})^{\top} \nabla \vec{m}\right], \label{eqn:fmhd-equations a} \\
    &\div \vec{v} = 0, \label{eqn:fmhd-equations b} \\
    &\partial_t \vec{B} + \eta \curl^2 \vec{B} - \curl(\vec{v} \times \vec{B}) = \vec{0}, \label{eqn:fmhd-equations c} \\
    &\div \vec{B} = 0, \label{eqn:fmhd-equations d}\\
    &\partial_t \vec{m} + (\vec{v} \cdot \nabla) \vec{m} = \gamma \vec{m} \times (\Delta \vec{m} + \vec{B}) - \chi \vec{m} \times (\vec{m} \times (\Delta \vec{m} + \vec{B})). \label{eqn:fmhd-equations e}
\end{alignat}
\end{subequations}
over the problem domain $(0,T) \times \Omega$. Here $T > 0$ and the set $\Omega$ is a bounded domain in $\R^3$ with smooth boundary $\partial \Omega$. The vector fields $\vec{v}, \vec{B}, \vec{m} : [0,T] \times \Omega \to \R^3$ represent the velocity of the fluid, the magnetic field, and the magnetisation, respectively, while the scalar field $p : [0,T] \times \Omega \to \R$ represents the pressure of the fluid. The constants $\mu, \eta$ and $\chi$ are positive parameters related to the viscosity the fluid, the diffusivity of the magnetic field, and the damping of the magnetisation. The parameter $\gamma$ is a non-zero constant representing the electron's gyromagnetic ratio. We set $\mu = \eta = \gamma = \chi = 1$ in \eqref{eqn:fmhd}, for the sake of simplicity.  This system is subjected to the boundary conditions
\begin{subequations} \label{eqn:boundary-conditions}
    \begin{alignat}{3}
    \vec{v} &= \vec{0}, \label{eqn:boundary-conditions a}\\
    \vec{B} \cdot \vec{n} &= 0, \label{eqn:boundary-conditions b}\\
    \curl \vec{B} \times \vec{n} &= \vec{0}, \label{eqn:boundary-conditions c}\\
    \frac{\partial \vec{m}}{\partial \vec{n}} &= \vec{0}, \label{eqn:boundary-conditions d}
    \end{alignat}
\end{subequations}
on $[0,T] \times \partial\Omega$, along with the initial data
\begin{equation*}
    \vec{v}(0,\cdot) = \vec{v}_0, \qquad \vec{B}(0,\cdot) = \vec{B}_0, \qquad \vec{m}(0,\cdot) = \vec{m}_0.
\end{equation*}

% magnetohydrodynamics
In 1942, Hannes Alfvén studied the properties of waves created by fluids that conducted electricity under the influence of a magnetic field \cite{alfven1942-article}. This scenario necessitates that the moving liquid exposed to a magnetic field creates an electric current within it which then influences the motion of the fluid as it interacts with the magnetic field. This field of inquiry, designated `magnetohydrodynamics', had tremendous applications in plasma research, solar physics, geophysics and more \cite{sheikholeslami2016-book}. The dynamics for an incompressible magnetohydrodynamical fluid were modelled by the following system of PDEs, consisting of the Navier-Stokes equation and the induction equation, derived from Maxwell's equations \cite{gerbeau2006-book}:
\begin{equation} \label{eqn:mhd-model}
    \begin{alignedat}{2}
        &\partial_t \vec{v} + (\vec{v} \cdot \nabla) \vec{v} - \mu\Delta \vec{v} + \nabla p = \curl \vec{B} \times \vec{B}, &\qquad &\div \vec{v} = 0, \\
        &\partial_t \vec{B} + \eta \curl^2 \vec{B} - \curl(\vec{v} \times \vec{B}) = \vec{0}, &\qquad &\div \vec{B} = 0.
    \end{alignedat}
\end{equation}
The intrinsic non-linearities of this system of PDEs, and its numerous applications have drawn a number of mathematicians to it. The equations of magnetohydrodynamics have been studied by mathematicians for decades. They have asked questions related to the existence and uniqueness of solutions, including weak-strong uniqueness, and the regularity of those solutions. A non-exhaustive list of their efforts includes \cite{chen2011-article, gunzburger1991-article, hu2010-article, li2011-article, sart2009-article, sermange1983-article, yan2013-article}.

% landau lifshitz equation
About a decade earlier, in 1935, Landau and Lifshitz studied the distribution of magnetic moments in a ferromagnetic crystal \cite{landau1935-article}, and based on phenomenological observations, produced the so-called Landau-Lifshitz (LL) equation
\begin{equation*}
    \partial_t \vec{m} = \vec{m} \times \vec{H}_{\mathrm{eff}} - \vec{m} \times (\vec{m} \times \vec{H}_{\mathrm{eff}}).
\end{equation*}
Here, $\vec{H}_{\mathrm{eff}}$ is the effective field, which consists of the terms representing the various energies that influence the dynamics of the magnetisation. Note that in \eqref{eqn:fmhd-equations e} we have $\vec{H}_{\mathrm{eff}} = \Delta \vec{m} + \vec{B}$, which corresponds to the exchange and Zeeman energies, respectively. Magnetisation fundamentally arises from quantum mechanical effects \cite{lakshmanan2011-article}, but the theory of micromagnetism attempts to capture those effects through a macroscopic lens by ignoring its atomic nature and looking at it through a more classical lens \cite{kruzik2006-article}. The purpose of this theory is to calculate the magnetisation of a material under the influence of varying effects such as differing fields, currents, inherent structures of materials, the interactions between magnetisations between different parts of the material and so on \cite{exl2020-book}. The LL equation is becoming increasingly relevant in industry as it assists in the development of electronic devices based on magnetic nano-wires, vortexes in nano-elements for sensors, magnetic recording, and so on \cite{cimrak2007-article}. The LL equations thus demonstrate a number of non-linear structures and has interesting geometric properties. This, in general, makes the equation highly complex and non-integrable \cite{lakshmanan2011-article}. For these reasons, mathematicians have been drawn to it and have studied it very closely over the last couple of decades. The first string of results were theoretical studies by mathematicians on the existence, uniqueness (or non-uniqueness) and regularity of solutions to the equations. A non-exhaustive list includes \cite{alouges1992-article, carbou2001-article, chen2000-article, chen1998-article, harpes2004-article, melcher2005-article, visintin1985-article, zhou1981-article}.

% fmhd equation

More recently, in 2015, Lingam \cite{lingam2015-article} studied the motion of conducting classical and quantum fluids that have spin, that is, they have an intrinsic magnetisation much like permanent magnets. Hence, in addition to being a magnetohydrodynamical model \eqref{eqn:mhd-model}, Lingam accounted for the dynamics of this intrinsic magnetisation. This model has potential applications in astrophysical and fusion plasmas, with some links to liquid crystals \cite{lingam2015-article}. Now, the resultant model consists of the Navier-Stokes equation for fluid flow, Maxwell's equation for magnetic fields and the Landau-Lifshitz equation for magnetisation and resembles the model \eqref{eqn:fmhd}. However, Lingam develops the model in the compressible case, whereas we look at the incompressible version, and we also leave out the entropy-density equation as it is decoupled from the others. Moreover, we differ from Lingam in that we also conserve the exchange energy of the magnetisation, represented by $|\nabla \vec{m}|^2$ in the Hamiltonian, the term $\div \left[(\nabla \vec{m})^{\top} (\nabla \vec{m})\right]$ in \eqref{eqn:fmhd-equations a}, and $\Delta \vec{m}$ in \eqref{eqn:fmhd-equations e}. The conservation of the exchange energy is necessary because it features prominently in other studies of the LL equations (see the references above). The resultant model \eqref{eqn:fmhd} thus consists of the Navier-Stokes equation for fluid flow, Maxwell's equation for magnetic fields and the Landau-Lifshitz equation for magnetisation.

Equations \eqref{eqn:fmhd} are similar to the ferrohydrodynamic model by Rinaldi and Zahn \cite{rinaldi2002-article} and Rosensweig \cite{rosensweig2013-book}, in that they both contain similar fluid and magnetisation dynamics. However, they differ in that \eqref{eqn:fmhd} deals with electrically conducting fluids and does not conserve angular momentum. Equations \eqref{eqn:fmhd} are also similar to the magnetoviscoelastic model \cite{benesova2018-article, forster2016-thesis} which looks at the dynamics of magnetorheological fluids \cite{rabinow1948-article}. This latter model is composed of the Navier-Stokes equation as well as the Landau-Lifshitz equation, but differs with the inclusion of the elastic deformation equation.

A mathematical inquisition into the existence, uniqueness and regularity of solutions to \eqref{eqn:fmhd} has not been performed. Here, we undertake the onus of an analytical study of these coupled equations.

The paper is organised as follows. In Section \ref{sec:prelim} we lay out any notation and useful results that we use in this paper along with the statements of the main results. We carry out the construction of the Galerkin solution, using the Faedo-Galerkin method, in Section \ref{sec:faedo-galerkin}. The proof of the existence (Theorem \ref{thm:existence}) and stability (Theorem \ref{thm:stability}) of solutions are laid out in Sections \ref{sec:existence} and \ref{sec:stability}, respectively.

\section{Preliminaries and Main Results} \label{sec:prelim}
% stokes operator results can be found here
% https://www.jstage.jst.go.jp/article/jmath1948/46/4/46_4_607/_pdf/-char/ja
% "Generalized resolvent estimates for the Stokes system in bounded and unbounded domains" by Reinhard Farwig and Hermann Sohr

% https://link.springer.com/content/pdf/10.1007/BF01200362.pdf
% "On the Semigroup of the Stokes Operator for Exterior Domains in LO-spaces" by Wolfgang Borchers and Hermann Sohr
% "On the Stokes operator in exterior domains" by Yoshikazu Giga and Hermann Sohr

% https://www.mathematik.tu-darmstadt.de/media/mathematik/forschung/preprint/preprints/2696.pdf
% Resolvent estimates of the Stokes system with Navier boundary conditions in general unbounded domains
%  Reinhard Farwig, Veronika Rosteck 

We first define the notation used in this paper. The set $\Omega \subset \R^3$ is a domain with a smooth boundary. The space $\LB^p(\Omega)$ denotes the space of $p$-integrable functions on $\Omega$ taking values in $\R^n$, where the value of $n$ is determined from context; throughout this paper, we use $\langle \cdot, \cdot\rangle$ to denote the $\LB^2(\Omega)$ inner product. Moreover, $\W^{m,p}(\Omega)$ denotes the usual Sobolev space of vector-valued functions, and we define $\HB^m(\Omega) := \W^{m,2}(\Omega)$. Furthermore, we define
\begin{align*}
    \HB_n^1(\Omega) &:= \{\vec{u} \in \HB^1(\Omega) : \vec{u} \cdot \vec{n} = 0 \text{ a.e. on } \partial\Omega\}, \\
    \HB_{\div}^1(\Omega) &:= \{\vec{u} \in \HB^1(\Omega) : \div \vec{u} = 0 \text{ a.e. on } \Omega\}, \\
    \HB_{0, \div}^1(\Omega) &:= \{\vec{u} \in \HB_0^1(\Omega) : \div \vec{u} = 0 \text{ a.e. on } \Omega\}, \\
    \HB_{n, \div}^1(\Omega) &:= \HB_n^1(\Omega) \cap \HB_{\div}^1(\Omega).
\end{align*}
Lastly, the spaces $L^p(0,T; X)$ and $W^{m,p}(0,T; X)$ denote the usual Bochner spaces of functions on $(0,T)$ taking values in a normed vector space $X$. Here we use the notation $\partial_t$ to refer to the time-derivative and $\partial_i$ as shorthand for $\partial_{x_i}$ to refer to the spatial derivatives; we also use $\partial_{\vec{n}}$ as shorthand for the directional derivative (see below). Moreover, for any vector-valued function $\vec{u} : \Omega \to \R^3$ we define
\begin{alignat*}{3}
    &\nabla \vec{u} : \Omega \to \R^{3 \times 3} & & \qquad \text{by} \qquad  \nabla \vec{u} := \pmat{\partial_1 \vec{u} & \partial_2 \vec{u} & \partial_3 \vec{u}} = \pmat{\partial_1 u_1 & \partial_2 u_1 & \partial_3 u_1 \\ \partial_1 u_2 & \partial_2 u_2 & \partial_3 u_2 \\ \partial_1 u_3 & \partial_2 u_3 & \partial_3 u_3}, \\[2ex]
    &\Delta \vec{u} : \Omega \to \R^3 & &\qquad \text{by} \qquad \Delta \vec{u} := \pmat{\Delta u_1 & \Delta u_2 & \Delta u_3}^{\top}, \\[2ex]
    &\displaystyle \frac{\partial \vec{u}}{\partial \vec{n}} : \partial\Omega \to \R^3 & &\qquad \text{by} \qquad \displaystyle\frac{\partial \vec{u}}{\partial \vec{n}} := (\nabla \vec{u}|_{\partial\Omega}) \vec{n}, \\[2ex]
    &\nabla^2 \vec{u} : \Omega \to \R^{3 \times 3 \times 3} & & \qquad \text{by} \qquad \nabla^2 \vec{u} := \nabla(\nabla \vec{u}),        
\end{alignat*}
where the latter denotes the tensor of second-order derivatives of $\vec{u}$ and not the Laplacian (which we denote by $\Delta$). Throughout this paper we use the constant $C$ to denote an arbitrary positive constant and further use the inequality notation $a \lsim b$ to imply that $a \leq C b$ for some $C$.

We now introduce the Leray projection $\PP : \LB^2(\Omega) \to \HB_{n,\div}^1(\Omega)$. The projection $\PP$ is defined using the Helmholtz-Leray decomposition of the space $\LB^2(\Omega)$. That is, for any $\vec{w} \in \LB^2(\Omega)$, we can decompose $\vec{w}$ such that
\begin{equation*}
    \vec{w} = \nabla p + \vec{u},
\end{equation*}
where $\vec{u} \in \HB_{n,\div}^1(\Omega)$ and $p \in \HB^1(\Omega)$. Here, $p$ solves the elliptic problem
\begin{align*}
    \Delta p &= \div \vec{w} \qquad \text{on } \Omega, \\
    \frac{\partial p}{\partial \vec{n}} &= \vec{w} \cdot \vec{n} \qquad \text{on } \partial\Omega.
\end{align*}
This decomposition is well-defined and unique (up to an additive constant for $p$). The Leray projection $\PP$ and the Stokes operator $A$ are defined by (see \cite[Section II.3]{foias2001-book} for more details)
\begin{equation} \label{eqn:stokes-operator}
\PP \vec{w} = \vec{u} \quad\forall\vec{w}\in\LB^2(\Omega)
\quad\text{and}\quad
A \vec{w} := -\PP \Delta \vec{w}
\quad\forall\vec{w}\in\HB^2(\Omega).
\end{equation}

The following identities are stated for the convenience of the reader. For $\vec{u}, \vec{v} \in \HB^1(\Omega)$, we have (see, e.g., \cite[Appendix B]{monk2003-book})
\begin{align}
    \curl(\vec{u} \times \vec{v}) &= \vec{u} (\div \vec{v}) - \vec{v}(\div \vec{u}) + (\vec{v} \cdot \nabla) \vec{u} - (\vec{u} \cdot \nabla) \vec{v}. \label{eqn:curl-cross-product}
\end{align}
and for $\vec{u} \in \HB^2(\Omega)$
\begin{equation} \label{eqn:curl-squared}
    \curl^2 \vec{u} = \nabla (\div \vec{u}) - \Delta \vec{u}.
\end{equation}
Moreover, we recall the integration by parts formula for $\curl$ \cite[Theorem 3.29]{monk2003-book}, namely that for any $\vec{u}, \vec{v} \in \HB^1(\Omega)$
\begin{equation} \label{eqn:curl-integration-by-parts}
    \liprod{\curl \vec{u}}{\vec{v}} = \liprod{\vec{u}}{\curl \vec{v}} + \langle \vec{u} \times \vec{v}, \vec{n} \rangle_{\LB^2(\partial\Omega)}
\end{equation}

\newcommand{\lnormD}[2]{\left\|#1\right\|_{\LB^{#2}(D)}}
\newcommand{\hnormD}[2]{\left\|#1\right\|_{\HB^{#2}(D)}}
\newcommand{\wnormD}[3]{\left\|#1\right\|_{\W^{#2,#3}(D)}}

We now state a few results that are crucial to our proof. The first two lemmas aid us in establishing the regularity of the magnetic field $\vec{B}$ and the magnetisation $\vec{m}$.

\begin{lemma} \label{lem:curl-estimate}
Let $D \subset \R^d$, $d = 2,3$, be a bounded domain with smooth boundary. If $\vec{u} \in \HB^1_{n,\div}(D)$ then
\begin{align*}
    \lnormD{\nabla \vec{u}}{2} &\lsim \lnormD{\curl \vec{u}}{2}
\end{align*}
and additionally if $\vec{u} \in \HB^2(D)$ and $\curl \vec{u} \times \vec{n} = \vec{0}$ on $\partial D$ we obtain
\begin{align*}
    \lnormD{\nabla^2 \vec{u}}{2} & \lsim \lnormD{\curl^2 \vec{u}}{2}.
\end{align*}
\end{lemma}
\begin{proof}
From  \cite[Appendix I Proposition 1.4]{temam2001-book} we have, for $m \geq 1$,
\begin{equation} \label{eqn:temam-curl-estimate}
    \hnormD{\vec{u}}{m} \lsim \lnormD{\vec{u}}{2} + \hnormD{\curl \vec{u}}{m-1} + \|{\div \vec{u}}\|_{H^{m-1}(D)} + \|\vec{u} \cdot \vec{n}\|_{H^{m-1/2}(\partial D)}.
\end{equation}
The first inequality follows by setting $m = 1$ and using our assumptions. To prove the second inequality, we first note that from \cite[(2.7)-(2.8)]{sermange1983-article} the problem
\begin{align*}
    -\Delta \vec{u} &= \vec{f}, \qquad \text{on } D, \\
    \vec{u} \cdot \vec{n} &= 0, \qquad \text{on } \partial D, \\
    \curl \vec{u} \times \vec{n} &= \vec{0}, \qquad \text{on } \partial D
\end{align*}
has the regularity
\begin{equation} \label{eqn:curl-curl-regularity}
    \hnormD{\vec{u}}{s+2} \lsim \hnormD{\vec{f}}{s}.
\end{equation}
But $\div \vec{u} = 0$ and \eqref{eqn:curl-squared} imply that
\begin{equation*}
    -\Delta \vec{u} = \curl^2 \vec{u},
\end{equation*}
and thus the result follows from setting $s=0$ in \eqref{eqn:curl-curl-regularity}.
\end{proof}

We borrow the following two lemmas from \cite{carbou2001-article} and \cite{chen2011-article}, respectively.

\begin{lemma} \label{lem:carbou}
Let $D \subset \R^d$, $d = 1,2,3$, be a bounded domain with smooth boundary. Then for all $\vec{u} \in \HB^2(D)$ such that $\displaystyle \frac{\partial \vec{u}}{\partial \vec{n}} = \vec{0}$ on $\partial D$,
\begin{align} \label{eqn:lem-carbou-eqn-1}
    \hnormD{\vec{u}}{2} &\lsim \left(\lnormD{\vec{u}}{2}^2 + \lnormD{\Delta \vec{u}}{2}^2\right)^{1/2},
\end{align}
and for $\vec{u} \in \HB^3(D)$ such that $\displaystyle \frac{\partial \vec{u}}{\partial \vec{n}} = \vec{0}$ on $\partial D$,
\begin{align}
    \hnormD{\nabla \vec{u}}{2} &\lsim \left(\lnormD{\nabla \vec{u}}{2}^2 + \lnormD{\Delta \vec{u}}{2}^2 + \lnormD{\nabla \Delta \vec{u}}{2}^2\right)^{1/2}, \label{eqn:lem-carbou-eqn-2} \\
    \lnormD{\nabla^2 \vec{u}}{3} &\lsim \hnormD{\vec{u}}{2} + \hnormD{\vec{u}}{2}^{1/2} \lnormD{\nabla \Delta \vec{u}}{2}^{1/2}. \label{eqn:lem-carbou-eqn-3}
\end{align}
\end{lemma}
\begin{proof}
Inequality \eqref{eqn:lem-carbou-eqn-1} results from the regularity of solutions of $-\Delta \vec{u} + \vec{u} = \vec{g}$ (for some $\vec{g} \in \LB^2(D)$) subject to the homogeneous Neumann boundary condition. This boundary condition also enables us to prove Inequality \eqref{eqn:lem-carbou-eqn-2} using Proposition~1.4 of Appendix~I in \cite{temam2001-book} with $\vec{v} = \nabla \vec{u}$. The last inequality \eqref{eqn:lem-carbou-eqn-3} issues from the Gagliardo-Nirenberg inequality
\begin{equation*}
    \wnormD{\vec{u}}{2}{3} \lsim \hnormD{\vec{u}}{2}^{1/2} \hnormD{\vec{u}}{3}^{1/2}
\end{equation*}
and \eqref{eqn:lem-carbou-eqn-2}.
\end{proof}

The next two lemmas establish estimates for certain terms that arise in our subsequent analysis.

\begin{lemma} \label{lem:chen}
For $\vec{f} \in \HB^1(\Omega)$ and $\vec{g} \in \HB^2(\Omega)$
\begin{equation*}
    \lnorm{|\vec{f}| |\nabla \vec{g}|}{2} \leq \hnorm{\vec{f}}{1} \hnorm{\vec{g}}{1}^{1/2} \hnorm{\vec{g}}{2}^{1/2}.
\end{equation*}
\end{lemma}
\begin{proof}
It follows from Hölder's inequality, the Sobolev embedding $\HB^1(\Omega) \hookrightarrow \LB^6(\Omega)$, and the Gagliardo-Nirenberg inequality that
\begin{align*}
    \lnorm{|\vec{f}| |\nabla \vec{g}|}{2} &\leq \lnorm{\vec{f}}{6} \lnorm{\nabla \vec{g}}{3} \\
    &\lsim \hnorm{\vec{f}}{1} \lnorm{\nabla \vec{g}}{2}^{1/2} \hnorm{\nabla \vec{g}}{1}^{1/2} \\
    &\lsim \hnorm{\vec{f}}{1} \hnorm{\vec{g}}{1}^{1/2} \hnorm{\vec{g}}{2}^{1/2},
\end{align*}
proving our result.
\end{proof}

\begin{lemma} \label{lem:curl-vxb}
For $\vec{u}, \vec{v}\in \HB_{\div}^1(\Omega)$ and $\vec{w} \in \LB^2(\Omega)$
\begin{equation*}
    \left|\liprod{\curl(\vec{u} \times \vec{v})}{\vec{w}}\right| \lsim \lnorm{|\vec{v}| \, |\nabla \vec{u}| \, |\vec{w}|}{1} + \lnorm{|\vec{u}| \, |\nabla \vec{v}| \, |\vec{w}|}{1}.
\end{equation*}
\end{lemma}
\begin{proof}
From \eqref{eqn:curl-cross-product} we know that
\begin{align*}
    \curl(\vec{u} \times \vec{v}) &= \vec{u} (\div \vec{v}) - \vec{v}(\div \vec{u}) + (\vec{v} \cdot \nabla) \vec{u} - (\vec{u} \cdot \nabla) \vec{v} \\
    &= (\vec{v} \cdot \nabla) \vec{u} - (\vec{u} \cdot \nabla) \vec{v}.
\end{align*}
Hence,
\begin{align*}
    \left|\liprod{\curl(\vec{u} \times \vec{v})}{\vec{w}}\right| &\lsim \left|\liprod{(\vec{v} \cdot \nabla) \vec{u}}{\vec{w}}\right| + \left|\liprod{(\vec{u} \cdot \nabla) \vec{v}}{\vec{w}}\right| \\
    &\lsim \lnorm{|\vec{v}| \, |\nabla \vec{u}| \, |\vec{w}|}{1} + \lnorm{|\vec{u}| \, |\nabla \vec{v}| \, |\vec{w}|}{1},
\end{align*}
thus completing the proof.
\end{proof}

It is a well-known result that the standard Landau-Lifshitz equation preserves the constraint $|\vec{m}| = 1$. We now show that this is also true of \eqref{eqn:fmhd-equiv e} which has the convective derivative $(\vec{v} \cdot \nabla) \vec{m}$. This will assist us in developing an equivalent formulation of \eqref{eqn:fmhd}.

\begin{proposition} \label{prop:implicit-unit-magnitude}
If the solutions $(\vec{v}, \vec{B}, \vec{m})$ are smooth, the magnetisation $\vec{m}$ satisfies $|\vec{m}| = 1$ on $[0,T] \times \Omega$ as soon as $|\vec{m}_0| = 1$.
\end{proposition}
\begin{proof}
Taking the inner product of \eqref{eqn:fmhd-equations e} with $\vec{m}$, we obtain
\begin{equation*}
    \frac{\partial}{\partial t} |\vec{m}|^2 + (\vec{v} \cdot \nabla) |\vec{m}|^2 = 0
\end{equation*}
Setting $b = |\vec{m}|^2 - 1$ we see that $b(t,\vec{x})$ satisfies the PDE
\begin{equation*}
    \frac{\partial b}{\partial t} + (\vec{v} \cdot \nabla) b = 0,
\end{equation*}
subject to the Neumann boundary condition $\partial b / \partial \vec{n} = 0$ and the initial condition $b(0,\cdot) = |\vec{m}_0|^2 - 1 = 0$. As $\div \vec{v} = 0$, by taking the $L^2(\Omega)$ inner-product with $b$ we have
\begin{align*}
    \frac{1}{2} \ddt \lnorm{b}{2}^2 = 0,
\end{align*}
implying that
\begin{equation*}
    \lnorm{b(t,\cdot)}{2}^2 = \lnorm{b(0,\cdot)}{2}^2 = 0.
\end{equation*}
This ends the proof.
\end{proof}

An equivalent formulation of \eqref{eqn:fmhd} can now be developed. Using Proposition \ref{prop:implicit-unit-magnitude} and the elementary identities,
\begin{align}
    \vec{u} \times (\vec{v} \times \vec{w}) &= (\vec{u} \cdot \vec{w}) \vec{v} - (\vec{u} \cdot \vec{v}) \vec{w}, \label{eqn:triple-cross-product} \\
    \Delta (|\vec{u}|^2) &= 2|\nabla \vec{u}|^2 + 2(\vec{u} \cdot \Delta \vec{u}), \label{eqn:laplace-magnitude}
\end{align}
and we have
\begin{equation*}
    \vec{m} \times (\vec{m} \times \Delta \vec{m}) = -\Delta \vec{m} -|\nabla \vec{m}|^2 \vec{m}.
\end{equation*}
Moreover, by expanding the last term in \eqref{eqn:fmhd-equations a} we obtain
\begin{align*}
    \div \left[(\nabla \vec{m})^{\top} \nabla \vec{m}\right] &= \frac{1}{2} \nabla |\nabla \vec{m}|^2 + (\nabla \vec{m})^{\top} \Delta \vec{m},
\end{align*}
thus giving us the equivalent problem
\begin{subequations} \label{eqn:fmhd-equiv}
    \begin{alignat}{2}
    &\partial_t \vec{v} + (\vec{v} \cdot \nabla) \vec{v} - \Delta \vec{v} + \nabla p = \curl \vec{B} \times \vec{B} + (\vec{B} \cdot \nabla) \vec{m} + \nabla \left(\vec{m} \cdot \vec{B} - \frac{1}{2} |\nabla \vec{m}|^2\right) \nonumber \\
    &\qquad \qquad \qquad \qquad \qquad \qquad \qquad - (\nabla \vec{m})^{\top} \Delta \vec{m}, \label{eqn:fmhd-equiv a} \\
    &\div \vec{v} = 0, \label{eqn:fmhd-equiv b} \\
    &\partial_t \vec{B} + \curl^2 \vec{B} - \curl(\vec{v} \times \vec{B}) = \vec{0}, \label{eqn:fmhd-equiv c} \\
    &\div \vec{B} = 0, \label{eqn:fmhd-equiv d}\\
    &\partial_t \vec{m} + (\vec{v} \cdot \nabla) \vec{m} - \Delta \vec{m} = |\nabla \vec{m}|^2 \vec{m} + \vec{m} \times (\Delta \vec{m} + \vec{B}) - \vec{m} \times (\vec{m} \times \vec{B}), \label{eqn:fmhd-equiv e} \\
    &|\vec{m}| = 1, \label{eqn:fmhd-equiv f}
\end{alignat}
\end{subequations}
on $(0,T) \times \Omega$.

We now define the notion of weak solutions to \eqref{eqn:fmhd-equiv}. We first take the dot-product of \eqref{eqn:fmhd-equiv a} and a test function $\boldsymbol{\varphi}$, integrate over $\Omega$ and (formally) use integration by parts assuming $\div \boldsymbol{\varphi} = 0$ on $\Omega$ and $\boldsymbol{\varphi} = 0$ on $\partial\Omega$, and rearrange to obtain
\begin{align} \label{eqn:formal-velocity}
\begin{split}
    \liprod{\partial_t \vec{v}}{\boldsymbol{\varphi}}  + \liprod{\nabla \vec{v}}{\nabla \boldsymbol{\varphi}} &=  -\liprod{(\vec{v} \cdot \nabla) \vec{v}}{\boldsymbol{\varphi}} + \liprod{\curl \vec{B} \times \vec{B}}{\boldsymbol{\varphi}} +\liprod{(\vec{B} \cdot \nabla) \vec{m}}{\boldsymbol{\varphi}} - \liprod{(\nabla \vec{m})^{\top} \Delta \vec{m}}{\vec{\varphi}}.
\end{split}
\end{align}
All the terms on the right-hand side are well-defined if $\vec{v}, \vec{B}, \boldsymbol{\varphi} \in \HB^1(\Omega)$ and $\vec{m} \in \HB^2(\Omega)$ since $\HB^1(\Omega) \hookrightarrow \LB^6(\Omega) \hookrightarrow \LB^3(\Omega)$ and $\lnorm{\vec{u}}{3} \lsim \lnorm{\vec{u}}{2}^{1/2} \hnorm{\vec{u}}{1}^{1/2}$ by the Gagliardo-Nirenberg inequality (Theorem \ref{thm:gagliardo-nirenberg}). Similarly, we take the dot-product of \eqref{eqn:fmhd-equiv c} and a test function $\boldsymbol{\omega}$, and integrating (using \eqref{eqn:curl-integration-by-parts} and \eqref{eqn:boundary-conditions c}) to obtain
\begin{align} \label{eqn:formal-magnetic-field}
    &\liprod{\partial_t \vec{B}}{\boldsymbol{\omega}} + \liprod{\curl \vec{B}}{\curl \boldsymbol{\omega}}  = \liprod{\curl(\vec{v} \times \vec{B})}{\boldsymbol{\omega}}.
\end{align}
Considering Lemma \ref{lem:curl-vxb}, the conditions on $\vec{v}$ and $\vec{B}$ mentioned previously do suffice, along with $\boldsymbol{\omega} \in \HB^1(\Omega)$. Lastly, we take the dot-product of \eqref{eqn:fmhd-equiv e} and a test function $\boldsymbol{\xi}$, and integrate (noting \eqref{eqn:boundary-conditions d}) to obtain
\begin{align} \label{eqn:formal-magnetisation}
\begin{split}
    \liprod{\partial_t \vec{m}}{\boldsymbol{\xi}} + \liprod{\nabla \vec{m}}{\nabla \boldsymbol{\xi}}  &= -\liprod{(\vec{v} \cdot \nabla) \vec{m}}{\boldsymbol{\varphi}} + \liprod{|\nabla \vec{m}|^2 \vec{m}}{\boldsymbol{\xi}} + \liprod{\vec{m} \times \nabla \vec{m}}{\nabla \boldsymbol{\xi}}  + \liprod{\vec{m} \times \vec{B}}{\boldsymbol{\xi}} \\
    &\qquad + \liprod{\vec{m} \times (\vec{m} \times \vec{B})}{\boldsymbol{\xi}}.
\end{split}
\end{align}
Once again, the conditions on $\vec{v}$ and $\vec{B}$ mentioned previously do suffice, with the additional requirement that $\boldsymbol{\xi} \in \HB^1(\Omega)$ and $\vec{m} \in \HB^2(\Omega)$, where the latter is required to accommodate for the additional non-linearity. This additional regularity is important because $\HB^2(\Omega) \hookrightarrow \LB^{\infty}(\Omega)$, which is clearly required, and because $\HB^1(\Omega) \hookrightarrow \LB^6(\Omega) \hookrightarrow \LB^4(\Omega)$, which is useful for the term  $\liprod{|\nabla \vec{m}|^2 \vec{m}}{\boldsymbol{\xi}}$. This motivates the following definition of solutions to the problem \eqref{eqn:fmhd-equiv}.

\begin{definition} \label{def:weak-solution}
Given $T > 0$ and $\vec{v}_0, \vec{B}_0, \vec{m}_0 \in \LB^2(\Omega)$, the functions $(\vec{v}, \vec{B}, \vec{m})$ are said to be a \textit{weak solution} to (\ref{eqn:fmhd-equiv}) if
\begin{enumerate}[i)]
    \item The velocity $\vec{v}$, magnetic field $\vec{B}$, and magnetisation $\vec{m}$ have the regularity
    \begin{align*}
        \vec{v} &\in L^2(0,T; \HB_{0,\div}^1(\Omega)), \\
        \vec{B} &\in L^2(0,T; \HB_{n,\div}^1(\Omega)), \\
        \vec{m} &\in L^2(0,T; \HB^2(\Omega)).
    \end{align*}

    \item The magnetisation $\vec{m}$ has magnitude $$|\vec{m}| = 1 \qquad \text{a.e. on } [0,T] \times \Omega.$$
 
    \item For all $(\boldsymbol{\varphi}, \boldsymbol{\omega}, \boldsymbol{\xi}) \in \HB_{0,\div}^1(\Omega) \times \HB^1(\Omega) \times \HB^1(\Omega)$, the triplet $(\vec{v}, \vec{B}, \vec{m})$ satisfies
    
    \begin{align} \label{eqn:weak-formulation-velocity}
        \begin{split}
            &\liprod{\vec{v}(t)}{\boldsymbol{\varphi}}  + \intOt{\liprod{\nabla \vec{v}(s)}{\nabla \boldsymbol{\varphi}}}   +\intOt{\liprod{(\vec{v}(s) \cdot \nabla) \vec{v}(s)}{\boldsymbol{\varphi}}} - \intOt{\liprod{\curl \vec{B}(s) \times \vec{B}(s)}{\boldsymbol{\varphi}}}\\
            &\qquad - \intOt{\liprod{(\vec{B}(s) \cdot \nabla) \vec{m}(s)}{\boldsymbol{\varphi}}} + \intOt{\liprod{(\nabla \vec{m}(s))^{\top} \Delta \vec{m}(s)}{\vec{\varphi}}} = \liprod{\vec{v}_0}{\boldsymbol{\varphi}},
        \end{split} \\
        \begin{split} \label{eqn:weak-formulation-magnetic-field}
            &\liprod{\vec{B}(t)}{\boldsymbol{\omega}} + \intOt{\liprod{\curl \vec{B}(s)}{\curl \boldsymbol{\omega}}}  - \intOt{\liprod{\curl(\vec{v}(s) \times \vec{B}(s))}{\boldsymbol{\omega}}} = \liprod{\vec{B}_0}{\boldsymbol{\omega}},
        \end{split} \\
        \begin{split} \label{eqn:weak-formulation-magnetisation}
            &\liprod{\vec{m}(t)}{\boldsymbol{\xi}} + \intOt{\liprod{\nabla \vec{m}(s)}{\nabla \boldsymbol{\xi}}} +\intOt{\liprod{(\vec{v}(s) \cdot \nabla) \vec{m}(s)}{\boldsymbol{\varphi}}} - \intOt{\liprod{|\nabla \vec{m}(s)|^2 \vec{m}(s)}{\boldsymbol{\xi}}} \\
            &\qquad - \intOt{\liprod{\vec{m}(s) \times \nabla \vec{m}(s)}{\nabla \boldsymbol{\xi}}}  - \intOt{\liprod{\vec{m}(s) \times \vec{B}(s)}{\boldsymbol{\xi}}} - \intOt{\liprod{\vec{m}(s) \times (\vec{m}(s) \times \vec{B}(s))}{\boldsymbol{\xi}}} \\
            &\qquad \qquad = \liprod{\vec{m}_0}{\boldsymbol{\xi}}.
        \end{split}
    \end{align}
\end{enumerate}
\end{definition}

The main results of our paper are as follows. The first theorem establishes the existence of solutions in the sense of Definition \ref{def:weak-solution} while the second theorem establishes the stability of those solutions.
\begin{theorem}[Existence] \label{thm:existence}
Let $\Omega \subset \R^3$ be a bounded domain with a smooth boundary and assume that the initial data $(\vec{v}_0, \vec{B}_0, \vec{m}_0)$ satisfies $\vec{v}_0 \in \HB_{0,\div}^1(\Omega) \cap \HB^2(\Omega)$, $\vec{B}_0 \in \HB_{n,\div}^1(\Omega) \cap \HB^2(\Omega)$ with $\curl \vec{B}_0 \times \vec{n} = \vec{0}$ on $\partial\Omega$, and $\vec{m}_0 \in \HB^3(\Omega)$ with $|\vec{m}_0| = 1$ on $\Omega$ and $\partial \vec{m}_0 / \partial \vec{n} = \vec{0}$  on $\partial\Omega$. Given $T > 0$, there exists $T^* \in(0,T]$ such that $(\vec{v}, \vec{B}, \vec{m})$ is a solution to \eqref{eqn:fmhd-equiv} on $[0,T^*] \times \Omega$ in the sense of Definition \ref{def:weak-solution} with the additional regularity
\begin{align*}
    \vec{v} &\in L^{\infty}(0,T^*; \HB^2(\Omega)) \cap L^2(0,T^*; \W^{2,6}(\Omega)) \cap C([0,T^*]; \HB^1(\Omega)) \cap W^{1,\infty}(0,T^*; \LB^2(\Omega)), \\
    \vec{B} &\in L^{\infty}(0,T^*; \HB^2(\Omega)) \cap L^2(0,T^*; \W^{2,6}(\Omega)) \cap  C([0,T^*]; \HB^1(\Omega)) \cap W^{1,\infty}(0,T^*; \LB^2(\Omega)), \\
    \vec{m} &\in L^{\infty}(0,T^*; \HB^3(\Omega)) \cap C([0,T^*]; \HB^2(\Omega)) \cap W^{1,\infty}(0,T^*; \HB^1(\Omega)).
\end{align*}
Note that in this case, $(\vec{v}, \vec{B}, \vec{m})$ satisfies \eqref{eqn:fmhd-equiv} for all $t \in [0,T^*]$ and almost all $\vec{x} \in \Omega$.
\end{theorem}

\begin{theorem}[Stability] \label{thm:stability}
Let $(\vec{v}_{1,0}, \vec{B}_{1,0}, \vec{m}_{1,0})$ and $(\vec{v}_{2,0}, \vec{B}_{2,0}, \vec{m}_{2,0})$ satisfy the assumptions for the initial data in Theorem \ref{thm:existence}, and let $(\vec{v}_1, \vec{B}_1, \vec{m}_1)$ and $(\vec{v}_2, \vec{B}_2, \vec{m}_2)$ be the two solutions associated with these initial data. If $\bar{\vec{v}} = \vec{v}_1 - \vec{v}_2$, $\bar{\vec{B}} = \vec{B}_1 - \vec{B}_2$, and $\bar{\vec{m}} = \vec{m}_1 - \vec{m}_2$, then there exists a constant $C$ independent of the initial data $(\bar{\vec{v}}_0, \bar{\vec{B}}_0, \bar{\vec{m}}_0)$ such that
\begin{equation*}
    \sup_{t \in [0,T']} \left(\lnorm{\bar{\vec{v}}(t)}{2} + \lnorm{\bar{\vec{B}}(t)}{2} + \hnorm{\bar{\vec{m}}(t)}{1}\right) \leq C \left( \lnorm{\bar{\vec{v}}_0}{2} + \lnorm{\bar{\vec{B}}_0}{2} + \hnorm{\bar{\vec{m}}_0}{1}\right),
\end{equation*}
where $T' = \min(T_1^*, T_2^*)$. Consequently, if $(\vec{v}_{1,0}, \vec{B}_{1,0}, \vec{m}_{1,0}) = (\vec{v}_{2,0}, \vec{B}_{2,0}, \vec{m}_{2,0})$ the solutions are unique.
\end{theorem}

\section{Faedo-Galerkin Approximation} \label{sec:faedo-galerkin}

In this section we use the Faedo-Galerkin approach to construct approximate solutions to \eqref{eqn:fmhd-equiv} and establish critical \textit{a priori} estimates for them.

% references for curlcurl eigenfunctions
% Teman R. Navier–Stokes Equations and Nonlinear Functional Analysis, CBMS-NSF Regional Conference Series in Applied Mathematics. SIAM: Philadelphia, 1983.
% lebris
% Some Mathematical Questions Related to the MHD Equations by Michel Sermange and Roger Temam
% follow the references in Comparison results for eigenvalues of curl curl operator and Stokes operator by Zhibing Zhang (https://link.springer.com/article/10.1007/s00033-018-0997-7)

\subsection{Eigenfunctions}
Let $\{\boldsymbol{\varphi}_i\}_{i=1}^{\infty}, \{\boldsymbol{\omega}_i\}_{i=1}^{\infty}$ and $\{\boldsymbol{\xi}_i\}_{i=1}^{\infty}$ denote the orthonormal bases of $\HB^1_{0,\div}(\Omega)$, $\HB_{n,\div}^1(\Omega)$ and $\LB^2(\Omega)$, respectively. They consist of the smooth eigenfunctions of the Stokes problem \cite[Section 2.6]{temam2001-book}
\begin{equation} \label{stokes-problem}
    \begin{cases}
        -\Delta \boldsymbol{\varphi}_i + \nabla p_i = \lambda_i \boldsymbol{\varphi} & \text{on } \Omega, \\
        \div \boldsymbol{\varphi}_i = 0 & \text{on } \Omega, \\
        \boldsymbol{\varphi}_i = 0 & \text{on } \partial\Omega,
    \end{cases}
\end{equation}
the steady state magnetic problem \cite[Section 2.2.2.2]{gerbeau2006-book}
\begin{equation} \label{steady-state-magnetic-problem}
    \begin{cases}
        \curl^2 \boldsymbol{\omega}_i = \mu_i \boldsymbol{\omega}_i & \text{on } \Omega, \\
        \boldsymbol{\omega}_i \cdot \vec{n} = 0 & \text{on } \partial\Omega, \\
        \curl \boldsymbol{\omega}_i \times \vec{n} = \vec{0} & \text{on } \partial \Omega,
    \end{cases}
\end{equation}
and the Laplacian problem \cite[Theorem 11.5.2]{jost2013-book}
\begin{equation} \label{laplacian-problem}
    \begin{cases}
        -\Delta \boldsymbol{\xi}_i = \sigma_i \boldsymbol{\xi}_i & \text{on } \Omega, \\[1ex]
        \displaystyle\frac{\partial \boldsymbol{\xi}_i}{\partial \vec{n}} = \vec{0} & \text{on } \partial\Omega.
    \end{cases}
\end{equation}
The eigenvalues $\lambda_i, \mu_i$ and $\sigma_i$ are all strictly positive and increasing as $i \to \infty$. For any~$k\in\N$, let
\begin{equation*}
    \V_k := \mathrm{span}\{\boldsymbol{\varphi}_1, \ldots, \boldsymbol{\varphi}_k\}, \qquad \XB_k := \mathrm{span}\{\boldsymbol{\omega}_1, \ldots, \boldsymbol{\omega}_k\}, \qquad \YB_k := \mathrm{span}\{\boldsymbol{\xi}_1, \ldots, \boldsymbol{\xi}_k\}.
\end{equation*}
We seek a Galerkin solution to (\ref{eqn:fmhd-equiv}) in within these spaces. That is we seek approximate solutions $(\vec{v}_k, \vec{B}_k, \vec{m}_k) \in \V_k \times \XB_k \times \YB_k$ of the form
\begin{equation*}
    \vec{v}_k(t,\vec{x}) = \sum_{i=1}^k a_i(t) \boldsymbol{\varphi}_i(\vec{x}), \qquad \vec{B}_k(t,\vec{x}) = \sum_{i=1}^k b_i(t) \boldsymbol{\omega}_i(\vec{x}), \qquad \vec{m}_k(t,\vec{x}) = \sum_{i=1}^k c_i(t) \boldsymbol{\xi}_i(\vec{x})
\end{equation*}
such that for all $(\boldsymbol{\varphi}, \boldsymbol{\omega}, \boldsymbol{\xi}) \in \V_k \times \XB_k \times \YB_k$
\begin{align}
\begin{split} \label{eqn:galerkin-velocity}
&\liprod{\partial_t \vec{v}_k}{\boldsymbol{\varphi}} + \liprod{\nabla \vec{v}_k}{\nabla \boldsymbol{\varphi}} = -\liprod{(\vec{v}_k \cdot \nabla)\vec{v}_k}{\boldsymbol{\varphi}} + \liprod{\curl \vec{B}_k \times \vec{B}_k}{\boldsymbol{\varphi}} + \liprod{(\vec{B}_k \cdot \nabla) \vec{m}_k}{\boldsymbol{\varphi}} \\
&\qquad \qquad \qquad \qquad \qquad \qquad  - \liprod{(\nabla \vec{m}_k)^{\top} \Delta \vec{m}_k}{\vec{\varphi}},
\end{split} \\[2ex]
&\liprod{\partial_t \vec{B}_k}{\boldsymbol{\boldsymbol{\omega}}} + \liprod{\curl \vec{B}_k}{\curl \boldsymbol{\omega}} = \liprod{\curl(\vec{v}_k \times \vec{B}_k)}{\boldsymbol{\omega}}, \label{eqn:galerkin-magnetic-field} \\[2ex]
\begin{split} \label{eqn:galerkin-magnetisation}
    &\liprod{\partial_t \vec{m}_k}{\boldsymbol{\xi}} + \liprod{\nabla \vec{m}_k}{\nabla \boldsymbol{\xi}} = -\liprod{(\vec{v}_k \cdot \nabla)\vec{m}_k}{\boldsymbol{\xi}} + \liprod{|\nabla \vec{m}_k|^2 \vec{m}_k}{\boldsymbol{\xi}} + \liprod{\vec{m}_k \times \Delta \vec{m}_k}{\boldsymbol{\xi}} \\
    &\qquad \qquad \qquad \qquad \qquad \qquad + \liprod{\vec{m}_k \times \vec{B}_k}{\boldsymbol{\xi}}  + \liprod{\vec{m}_k \times (\vec{m}_k \times \vec{B}_k)}{\boldsymbol{\xi}}.
\end{split}
\end{align}
Let $\{(\vec{v}_0^k, \vec{B}_0^k, \vec{m}_0^k)\}_{k=1}^{\infty} \subset \V_k \times \XB_k \times \YB_k$ be a sequence of functions such that $(\vec{v}_0^k, \vec{B}_0^k) \to (\vec{v}_0, \vec{B}_0)$ in $\HB^2(\Omega)$ and $\vec{m}_0^k \to \vec{m}_0$ in $\HB^3(\Omega)$. Choose $\{a_i(0)\}_{i=1}^k$, $\{b_i(0)\}_{i=1}^k$, and $\{c_i(0)\}_{i=1}^k$ such that
\begin{equation} \label{eqn:galerkin-initial-data}
    \vec{v}_k(0,\cdot) = \vec{v}_0^k, \qquad \vec{B}_k(0,\cdot) = \vec{B}_0^k, \qquad \vec{m}_k(0,\cdot) = \vec{m}_0^k.
\end{equation}
Let $\vec{z}(t) = (\vec{a}(t), \vec{b}(t), \vec{c}(t))^{\top}$ i.e., the vector collection of the time-dependent coefficients of $\vec{v}_k, \vec{B}_k$ and $\vec{m}_k$. Then as the eigenvectors are orthonormal bases, (\ref{eqn:galerkin-velocity})--(\ref{eqn:galerkin-magnetisation}) amounts to finding $\vec{z}$ such that
\begin{equation} \label{eqn:galerkin-ode}
    \vec{z}'(t) = \vec{F}(\vec{z}(t)),
\end{equation}
where $\vec{F}$ is a vector-valued multivariate polynomial. Here, the initial values $\vec{z}(0)$ are given by the approximate initial data $(\vec{v}_0^k, \vec{B}_0^k, \vec{m}_0^k)$. As $\vec{F}$ is a multivariate polynomial in each of its components, it is locally Lipschitz and thereby the Cauchy-Lipschitz theorem guarantees the local existence and uniqueness of a solution $\vec{z}$. Hence, there exists $\vec{v}_k, \vec{B}_k$ and $\vec{m}_k$ that solves the weak formulation. Moreover, as $\vec{F}$ is smooth $\vec{z}(t)$ is also smooth in time.

Given that approximate solutions exist, we now seek to bound this sequence of solutions in certain spaces.

\subsection{\textit{A Priori} Estimates}

For simplicity of notation we let
\begin{align} \label{eqn:j-def}
    \begin{split}
        J(t) &:= \hnorm{\vec{v}_k(t)}{2}^2 + \hnorm{\vec{B}_k(t)}{2}^2 + \hnorm{\vec{m}_k(t)}{3}^2 \\
        &\qquad + \lnorm{\partial_t \vec{v}_k(t)}{2}^2 + \lnorm{\partial_t \vec{B}_k(t)}{2}^2 + \hnorm{\partial_t \vec{m}_k(t)}{1}^2,
    \end{split}
\end{align}
and will often abbreviate $J(t)$ as $J$ below.

The following lemmas establish precursory inequalities that will be collated together at the end of this section to produce estimates for the Galerkin solution. The first two lemmas primarily deal with the $\HB^1(\Omega)$ norms of the Galerkin solutions.

\begin{lemma} \label{lem:first-order-estimate-1}
For each $k \in \N$ we have
\begin{align*}
    \ddt \lnorm{\vec{v}_k}{2}^2 + \lnorm{\nabla \vec{v}_k}{2}^2 &\lsim J + J^2, \\
    \ddt \lnorm{\vec{B}_k}{2}^2 + \lnorm{\curl \vec{B}_k}{2}^2 &\lsim J^{3/2}, \\
    \ddt \lnorm{\vec{m}_k}{2}^2 + \lnorm{\nabla \vec{m}_k}{2}^2 &\lsim J^2.
\end{align*}
\end{lemma}
\begin{proof}
It is noted that these estimates hold for $J(t)$ with $\hnorm{\vec{m}_k}{2}^2$ only. However, we use $\hnorm{\vec{m}_k}{3}^2$ as this will be required in subsequent lemmas. Setting $\boldsymbol{\varphi} = \vec{v}_k$ in (\ref{eqn:galerkin-velocity}), and using Young's inequality and the Sobolev embedding $\HB^2(\Omega) \hookrightarrow \LB^{\infty}(\Omega)$ we have
\begin{align*}
    \frac{1}{2} \ddt \lnorm{\vec{v}_k}{2}^2 + \lnorm{\nabla \vec{v}_k}{2}^2 &\lsim \left|\liprod{(\nabla \vec{v}_k) \vec{v}_k}{\vec{v}_k}\right| + \lnorm{|\curl \vec{B}_k| \, |\vec{B}_k| \, |\vec{v}_k|}{1} + \lnorm{|\vec{B}_k| \, |\nabla \vec{m}_k| \, |\vec{v}_k|}{1} \\
    &\qquad + \lnorm{|\nabla \vec{m}| \, |\Delta \vec{m}| \, |\vec{v}_k|}{1} \\
    &\lsim \lnorm{\vec{B}_k}{\infty} \lnorm{\curl \vec{B}_k}{2} \lnorm{\vec{v}_k}{2} \\
    &\qquad + \lnorm{\vec{B}_k}{\infty} \lnorm{\nabla \vec{m}_k}{2} \lnorm{\vec{v}_k}{2} \\
    &\qquad + \lnorm{\vec{v}_k}{\infty} \lnorm{\nabla \vec{m}_k}{2} \lnorm{\Delta \vec{m}_k}{2} \\
    &\lsim \hnorm{\vec{v}}{2}^2 + \hnorm{\vec{B}_k}{2}^4 + \hnorm{\vec{m}_k}{2}^4 \\
    &\lsim J + J^2.
\end{align*}
Here, the term involving $(\vec{v}_k \cdot \nabla) \vec{v}_k$ vanishes because $\vec{v}_k$ is divergence free. The first inequality is thereby proved. The second inequality follows similarly by setting $\boldsymbol{\omega} = \vec{B}_k$ in (\ref{eqn:galerkin-magnetic-field}), integrating by parts and using Lemma \ref{lem:curl-vxb} to obtain
\begin{align*}
    \frac{1}{2} \ddt \lnorm{\vec{B}_k}{2}^2 + \lnorm{\curl \vec{B}_k}{2}^2 &\lsim \lnorm{|\vec{v}_k| |\nabla \vec{B}_k| |\vec{B}_k|}{1} + \lnorm{|\vec{B}_k| |\nabla \vec{v}_k| |\vec{B}_k|}{1} \\
    &\lsim \lnorm{\vec{B}_k}{\infty} \lnorm{\nabla \vec{v}_k}{2} \lnorm{\vec{B}_k}{2} \\
    &\qquad + \lnorm{\vec{B}_k}{\infty} \lnorm{\vec{v}_k}{2} \lnorm{\nabla \vec{B}_k}{2} \\
    &\lsim \hnorm{\vec{B}_k}{2}^2 \lnorm{\nabla \vec{v}_k}{2} + \hnorm{\vec{B}_k}{2}^2 \lnorm{\vec{v}_k}{2} \\
    &\lsim \hnorm{\vec{v}_k}{2}^3 + \hnorm{\vec{B}_k}{2}^3 \\
    &\lsim J^{3/2}.
\end{align*}
Finally, the third inequality follows by also setting $\boldsymbol{\xi} = \vec{m}_k$ in (\ref{eqn:galerkin-magnetisation}) to obtain
\begin{align*}
    \frac{1}{2} \ddt \lnorm{\vec{m}_k}{2}^2 + \lnorm{\nabla \vec{m}_k}{2}^2 \lsim \lnorm{\vec{m}_k}{\infty}^2 \lnorm{\vec{m}_k}{2}^2 \lsim \hnorm{\vec{m}_k}{2}^4 \lsim J^2.
\end{align*}
Once again, the term involving $(\vec{v}_k \cdot \nabla) \vec{m}_k$ vanishes because $\vec{v}_k$ is divergence free.
\end{proof}

\begin{lemma} \label{lem:first-order-estimate-2}
For each $k \in \N$ we have
\begin{align*}
     \ddt \lnorm{\nabla \vec{v}_k}{2} + \lnorm{\partial_t \vec{v}_k}{2}^2 &\lsim J + J^2, \\
     \ddt \lnorm{\curl \vec{B}_k}{2}+\lnorm{\partial_t \vec{B}_k}{2}^2  &\lsim J + J^2.
\end{align*}
\end{lemma}
\begin{proof}
Here we do the same as in Lemma \ref{lem:first-order-estimate-1} and deal with the non-linear terms similarly. First, we set $\boldsymbol{\varphi} = \partial_t \vec{v}_k$ in (\ref{eqn:galerkin-velocity}) and use Young's inequality and the Sobolev embedding $\HB^2(\Omega) \hookrightarrow \LB^{\infty}(\Omega)$ to deduce that
\begin{align*}
    \lnorm{\partial_t \vec{v}_k}{2}^2 + \frac{1}{2} \ddt \lnorm{\nabla \vec{v}_k}{2} &\lsim \lnorm{\vec{v}_k}{\infty} \lnorm{\nabla \vec{v}_k}{2} \lnorm{\partial_t \vec{v}_k}{2} \\
    &\qquad + \lnorm{\vec{B}_k}{\infty} \lnorm{\curl \vec{B}_k}{2} \lnorm{\partial_t \vec{v}_k}{2} \\
    &\qquad + \lnorm{\vec{B}_k}{\infty} \lnorm{\nabla \vec{m}_k}{2} \lnorm{\partial_t \vec{v}_k}{2} \\
    &\qquad + \lnorm{\nabla \vec{m}_k}{\infty} \lnorm{\Delta \vec{m}_k}{2} \lnorm{\partial_t \vec{v}_k}{2} \\
    &\lsim \hnorm{\vec{v}_k}{2}^4 + \hnorm{\vec{B}_k}{2}^4 + \hnorm{\vec{m}_k}{3}^4 + \lnorm{\partial_t \vec{v}_k}{2}^2 \\
    &\lsim J + J^2.
\end{align*}
For the second inequality, we set $\boldsymbol{\omega} = \partial_t \vec{B}_k$ in \eqref{eqn:galerkin-magnetic-field} and use Lemma \ref{lem:curl-vxb} and Young's inequality to obtain
\begin{align*}
    \lnorm{\partial_t \vec{B}_k}{2}^2 + \frac{1}{2} \ddt \lnorm{\curl \vec{B}_k}{2} &\lsim \lnorm{|\vec{v}_k| |\nabla \vec{B}_k| |\partial_t \vec{B}_k|}{1} + \lnorm{|\vec{B}_k| |\nabla \vec{v}_k| |\partial_t \vec{B}_k|}{1} \\
    &\lsim \lnorm{\vec{B}_k}{\infty} \lnorm{\nabla \vec{v}_k}{2} \lnorm{\partial_t \vec{B}_k}{2} \\
    &\qquad + \lnorm{\vec{v}_k}{\infty} \lnorm{\nabla \vec{B}_k}{2} \lnorm{\partial_t \vec{B}_k}{2} \\
    &\lsim \hnorm{\vec{v}_k}{2}^4 + \hnorm{\vec{B}_k}{2}^4 + \lnorm{\partial_t \vec{B}_k}{2}^2 \\
    &\lsim J + J^2,
\end{align*}
thereby proving the second inequality.
\end{proof}

The solution $\vec{m}_k$ of \eqref{eqn:galerkin-magnetisation} requires stronger estimates on its $\HB^2$-norm, which is proved below.

\begin{lemma} \label{lem:second-order-estimate-magnetisation}
    For each $k \in \N$ we have
    \begin{align*}
        \ddt \lnorm{\Delta \vec{m}_k}{2}^2 + \lnorm{\nabla \Delta \vec{m}_k}{2}^2 &\lsim J + J^3.
    \end{align*}
    \end{lemma}
    \begin{proof}
    Setting $\boldsymbol{\xi} = \Delta^2 \vec{m}_k$ in (\ref{eqn:galerkin-magnetisation}) we see that
    \begin{align*}
        \frac{1}{2} \ddt \lnorm{\Delta \vec{m}_k}{2}^2 + \lnorm{\nabla \Delta \vec{m}_k}{2}^2 &= \liprod{\nabla \Big((\vec{v}_k \cdot \nabla)\vec{m}_k\Big)}{\nabla \Delta \vec{m}_k}  - \liprod{\nabla \Big(|\nabla \vec{m}_k|^2 \vec{m}_k\Big)}{\nabla \Delta \vec{m}_k} \\
        &\qquad - \liprod{\nabla\Big(\vec{m}_k \times \Delta \vec{m}_k\Big)}{\nabla \Delta \vec{m}_k} - \liprod{\nabla(\vec{m}_k \times \vec{B}_k\Big)}{\nabla \Delta \vec{m}_k} \\
        &\qquad - \liprod{\nabla \Big(\vec{m}_k \times (\vec{m}_k \times \vec{B}_k)\Big)}{\nabla \Delta \vec{m}_k} \\
        &\lsim |I_1| + \cdots + |I_5|,
    \end{align*}
    where each $I_i$ represents the corresponding inner product in the equation above. We produce these estimates by Young's inequality and the Sobolev embedding theorem. We estimate $I_1$ using the Sobolev embeddings $\HB^1(\Omega) \hookrightarrow \LB^6(\Omega) \hookrightarrow \LB^3(\Omega)$ and $\HB^2(\Omega) \hookrightarrow \LB^{\infty}(\Omega)$ to obtain
    \begin{align*}
        |I_1| &\lsim \lnorm{\nabla \vec{m}_k}{\infty} \lnorm{\nabla \vec{v}_k}{2} \lnorm{\nabla \Delta \vec{m}_k}{2} + \lnorm{\vec{v}_k}{\infty} \lnorm{\nabla^2 \vec{m}_k}{2} \lnorm{\nabla \Delta \vec{m}_k}{2} \\
        &\lsim \hnorm{\vec{m}_k}{3}^2 \hnorm{\vec{v}_k}{1} + \hnorm{\vec{v}_k}{2} \hnorm{\vec{m}_k}{3}^2 \\
        &\lsim J + J^2.
    \end{align*}
    For $I_2$, we use the same Sobolev embeddings so that
    \begin{align*}
        |I_2| &\lsim \lnorm{\vec{m}_k}{\infty} \lnorm{\nabla \vec{m}_k}{\infty} \lnorm{\nabla^2 \vec{m}_k}{2} \lnorm{\nabla \Delta \vec{m}_k}{2} \\
        &\qquad + \lnorm{\nabla \vec{m}_k}{\infty}^2 \lnorm{\nabla \vec{m}_k}{2} \lnorm{\nabla \Delta \vec{m}_k}{2} \\
        &\lsim \hnorm{\vec{m}_k}{3}^4 + \hnorm{\vec{m}_k}{3}^4 \\
        &\lsim J^2.
    \end{align*}
    The expansion of the product rule for the cross-product terms are fairly obvious; we follow the same strategy as before to produce
    \begin{equation*}
        |I_3| \lsim J^{3/2}, \qquad |I_4| \lsim J + J^2, \qquad |I_5| \lsim J + J^3.
    \end{equation*}
    The result is obtained by summing up the estimates.
\end{proof}

Due to the non-linearity of the equations in \eqref{eqn:fmhd-equiv} we will require stronger estimates on the first order time derivatives of $\vec{v}_k$, $\vec{B}_k$ and $\vec{m}_k$. We accomplish this by differentiating the Galerkin formulations \eqref{eqn:galerkin-velocity}--\eqref{eqn:galerkin-magnetisation} with respect to time. The next four lemmas are concerned with these terms.

\begin{lemma} \label{lem:time-derivative-estimate-velocity}
Let $\varepsilon > 0$. For each $k \in \N$ we have
\begin{align*}
    \ddt \lnorm{\partial_t \vec{v}_k}{2}^2 + \lnorm{\partial_t \nabla \vec{v}_k}{2}^2 &\lsim J + J^3 + \varepsilon \lnorm{\partial_t \nabla \vec{B}_k}{2}^2 + \varepsilon \lnorm{\partial_t \Delta \vec{m}_k}{2}^2.
\end{align*}
\end{lemma}
\begin{proof}
As each component of $\vec{F}$ in \eqref{eqn:galerkin-ode} is smooth, we conclude that $\vec{z} \in C^2(0,T^*)$ by recursion. Hence, differentiating (\ref{eqn:galerkin-velocity})--(\ref{eqn:galerkin-magnetisation}) with respect to time is permissible. Differentiating (\ref{eqn:galerkin-velocity}), we obtain
\begin{align} \label{eqn:diff-galerkin-velocity}
\begin{split}
    \liprod{\partial^2_t \vec{v}_k}{\boldsymbol{\varphi}} + \liprod{\partial_t \nabla \vec{v}_k}{\nabla \boldsymbol{\varphi}} &= -\liprod{(\partial_t\vec{v}_k \cdot \nabla)\vec{v}_k} {\boldsymbol{\varphi}} - \liprod{(\vec{v}_k \cdot \nabla)\partial_t\vec{v}_k} {\boldsymbol{\varphi}} + \liprod{\partial_t \curl \vec{B}_k \times \vec{B}_k}{\boldsymbol{\varphi}} \\
    &\qquad + \liprod{\curl \vec{B}_k \times \partial_t \vec{B}_k}{\boldsymbol{\varphi}} + \liprod{(\partial_t \vec{B}_k \cdot \nabla) \vec{m}_k}{\boldsymbol{\varphi}} + \liprod{(\vec{B}_k \cdot \nabla) \partial_t \vec{m}_k}{\boldsymbol{\varphi}} \\
    &\qquad - \liprod{(\partial_t \nabla \vec{m}_k)^{\top} \Delta \vec{m}_k}{\vec{\varphi}} - \liprod{(\nabla \vec{m}_k)^{\top} \partial_t \Delta \vec{m}_k}{\vec{\varphi}}.
\end{split}
\end{align}
Setting $\boldsymbol{\varphi} = \partial_t \vec{v}_k$ we obtain
\begin{align*}
    \frac{1}{2} \ddt \lnorm{\partial_t \vec{v}_k}{2} + \lnorm{\partial_t \nabla \vec{v}_k}{2}^2 &\lsim |I_1| + \cdots + |I_8|,
\end{align*}
where each $I_i$ represents the corresponding inner product on the right-hand side of (\ref{eqn:diff-galerkin-velocity}). We deal with the non-linear terms in a similar manner to Lemma \ref{lem:first-order-estimate-1}. First, we estimate $I_1$ by using Young's inequality, the Sobolev embedding $\HB^1(\Omega) \hookrightarrow \LB^6(\Omega)$ and the Gagliardo-Nirenberg inequality as follows
\begin{align*}
    |I_1| &\lsim \lnorm{\nabla \vec{v}_k}{2} \lnorm{\partial_t \vec{v}_k}{3} \lnorm{\partial_t \vec{v}_k}{6} \\
    &\lsim \lnorm{\nabla \vec{v}_k}{2} \lnorm{\partial_t \vec{v}_k}{2}^{1/2} \hnorm{\partial_t \vec{v}_k}{1}^{1/2} \hnorm{\partial_t \vec{v}_k}{1} \\
    &\lsim \hnorm{\vec{v}_k}{1}^6 + \lnorm{\partial_t \vec{v}_k}{2}^6 + \varepsilon \hnorm{\partial_t \vec{v}_k}{1}^2 \\
    &\lsim J + J^3 + \varepsilon \lnorm{\partial_t \nabla \vec{v}_k}{2}^2.
\end{align*}
It is noted that $I_2 = 0$ since $\vec{v}_k$ is divergence-free. For the inner product $I_3$, we use the Sobolev embedding $\HB^2(\Omega) \hookrightarrow \LB^{\infty}(\Omega)$ to obtain
\begin{align*}
    |I_3| &\lsim \lnorm{\vec{B}_k}{\infty}  \lnorm{\partial_t \nabla \vec{B}_k}{2} \lnorm{\partial_t \vec{v}_k}{2} \\
    &\lsim \hnorm{\vec{B}_k}{2}^4 + \lnorm{\partial_t \vec{v}_k}{2}^4 + \varepsilon \lnorm{\partial_t \nabla \vec{B}_k}{2}^2 \\
    &\lsim J^2 + \varepsilon \lnorm{\partial_t \nabla \vec{B}_k}{2}^2.
\end{align*}
Now, the estimate for $I_4$ follows in the same manner as $I_1$ so that
\begin{align*}
    |I_4| &\lsim J + J^3 + \varepsilon \lnorm{\partial_t \nabla \vec{B}_k}{2}^2 + \varepsilon \lnorm{\partial_t \nabla \vec{v}_k}{2}^2,
\end{align*}
while the estimates for $I_5$ and $I_6$ follow the procedure for $I_3$ to give
\begin{align*}
    |I_5| + |I_6| &\lsim J + J^2.
\end{align*}
Lastly, the inner product $I_7$ follows in the same manner as $I_1$, along with the additional embedding $\LB^6(\Omega) \hookrightarrow \LB^3(\Omega)$ to produce
\begin{align*}
    |I_7| &\lsim \lnorm{\partial_t \nabla \vec{m}_k}{2} \lnorm{\Delta \vec{m}_k}{3} \lnorm{\partial_t \vec{v}_k}{6} \\
    &\lsim \hnorm{\partial_t \vec{m}_k}{1} \hnorm{\vec{m}_k}{3} \lnorm{\partial_t \nabla \vec{v}_k}{2} \\
    &\lsim \hnorm{\partial_t \vec{m}_k}{1}^4 + \hnorm{\vec{m}_k}{3}^4 + \varepsilon \lnorm{\partial_t \nabla \vec{v}_k}{2}^2 \\
    &\lsim J^2 + \varepsilon \lnorm{\partial_t \nabla \vec{v}_k}{2}^2.
\end{align*}
We estimate for $I_8$ similarly to $I_3$ and obtain
\begin{align*}
    |I_8| &\lsim J^2 + \varepsilon \lnorm{\partial_t \Delta \vec{m}_k}{2}^2.
\end{align*}
The required result is obtained by summing up these estimates and absorbing the $\varepsilon \lnorm{\partial_t \nabla \vec{v}_k}{2}^2$ term into the left-hand side.
\end{proof}

\begin{lemma} \label{lem:time-derivative-estimate-magnetic-field}
Let $\varepsilon > 0$. For each $k \in \N$ we have
\begin{align*}
        \ddt \lnorm{\partial_t \vec{B}_k}{2}^2 + \lnorm{\partial_t \nabla \vec{B}_k}{2}^2 &\lsim J + J^3 + \varepsilon \lnorm{\partial_t \nabla \vec{v}_k}{2}^2.
\end{align*}
\end{lemma}
\begin{proof}
As with Lemma \ref{lem:time-derivative-estimate-velocity}, we differentiate (\ref{eqn:galerkin-magnetic-field}) with respect to time to obtain
\begin{align} \label{eqn:diff-galerkin-magnetic-field}
    \liprod{\partial^2_t \vec{B}_k}{\boldsymbol{\boldsymbol{\omega}}} + \liprod{\partial_t \curl \vec{B}_k}{\curl\boldsymbol{\omega}} &= \liprod{\curl(\partial_t \vec{v}_k \times \vec{B}_k)}{\boldsymbol{\omega}} + \liprod{\curl(\vec{v}_k \times \partial_t \vec{B}_k)}{\boldsymbol{\omega}}.
\end{align}
Setting $\boldsymbol{\omega} = \partial_t \vec{B}_k$ we have
\begin{align*}
    \frac{1}{2} \ddt \lnorm{\partial_t \vec{B}_k}{2}^2 + \lnorm{\partial_t \curl \vec{v}_k}{2}^2 &\lsim |I_1| + |I_2|,
\end{align*}
where each $I_i$ represents the corresponding inner product in (\ref{eqn:diff-galerkin-magnetic-field}). Using Lemma \ref{lem:curl-vxb} and then Young's inequality, the Gagliardo-Nirenberg inequality and the Sobolev embedding $\HB^1(\Omega) \hookrightarrow \LB^6(\Omega)$ in a similar manner to the proof of Lemma \ref{lem:time-derivative-estimate-velocity}, we have the following bounds:
\begin{align*}
    |I_1| &\lsim \lnorm{|\vec{B}_k| \, |\partial_t \nabla \vec{v}_k| \, |\partial_t \vec{B}_k|}{1} + \lnorm{|\partial_t \vec{v}_k| \, |\nabla \vec{B}_k| \, |\partial_t \vec{B}_k|}{1} \\
    &\lsim \hnorm{\vec{B}_k}{1}^6 + \lnorm{\partial_t \vec{B}_k}{2}^6 + \lnorm{\partial_t \vec{v}_k}{2}^2 + \varepsilon \hnorm{\partial_t \vec{B}_k}{1}^2 + \varepsilon \lnorm{\partial_t \nabla \vec{v}_k}{2}^2 \\
    &\lsim J + J^3 + \varepsilon \lnorm{\partial_t \nabla \vec{B}_k}{2}^2 + \varepsilon \lnorm{\partial_t \nabla \vec{v}_k}{2}^2, \\[2ex]
    |I_2| &\lsim \lnorm{|\partial_t \vec{B}_k| |\nabla \vec{v}_k| |\partial_t \vec{B}_k|}{1} + \lnorm{|\vec{v}_k| |\partial_t \nabla \vec{B}_k| |\partial_t \vec{B}_k|}{1} \\
    &\lsim \lnorm{\nabla \vec{v}_k}{3} \lnorm{\partial_t \vec{B}_k}{3}^2 + \lnorm{\vec{v}_k}{6} \lnorm{\partial_t \vec{B}_k}{3} \lnorm{\partial_t \nabla \vec{B}_k}{2} \\
    &\lsim \hnorm{\vec{v}_k}{1}^4 + \lnorm{\partial_t \vec{B}_k}{2}^4 + \varepsilon \hnorm{\partial_t \vec{B}_k}{1}^2 + \hnorm{\vec{v}_k}{1}^6 + \lnorm{\partial_t \vec{B}_k}{2}^6 \\
    &\lsim J + J^2 + J^3 + \varepsilon \lnorm{\partial_t \nabla \vec{B}_k}{2}^2 \\
    &\lsim J + J^3 + \varepsilon \lnorm{\partial_t \nabla \vec{B}_k}{2}^2,
\end{align*}
The result is proved by summing up these estimates, using Lemma \ref{lem:curl-estimate} and absorbing the term $\varepsilon \lnorm{\partial_t \nabla \vec{B}_k}{2}^2$ into the left-hand side.
\end{proof}

\begin{lemma} \label{lem:time-derivative-estimate-magnetisation-1}
Let $\varepsilon > 0$. For each $k \in \N$ we have
\begin{align*}
    \ddt \lnorm{\partial_t \vec{m}_k}{2}^2 + \lnorm{\partial_t \nabla \vec{m}_k}{2}^2 &\lsim J + J^3 + \varepsilon \lnorm{\partial_t \Delta \vec{m}_k}{2}^2.
\end{align*}
\end{lemma}
\begin{proof}
Once again, we differentiate (\ref{eqn:galerkin-magnetisation}) with respect to time to obtain
\begin{align} \label{eqn:diff-galerkin-magnetisation}
\begin{split}
    \liprod{\partial_t^2 \vec{m}_k}{\boldsymbol{\xi}} + \liprod{\partial_t \nabla \vec{m}_k}{\nabla \boldsymbol{\xi}} &= -\liprod{(\partial_t \vec{v}_k \cdot \nabla)\vec{m}_k}{\boldsymbol{\xi}} -\liprod{(\vec{v}_k \cdot \nabla)\partial_t \vec{m}_k}{\boldsymbol{\xi}} + 2\liprod{(\partial_t \nabla \vec{m}_k \cdot \nabla \vec{m}_k) \vec{m}_k}{\boldsymbol{\xi}}\\
    &\qquad + \liprod{|\nabla \vec{m}_k|^2 \partial_t \vec{m}_k}{\boldsymbol{\xi}} + \liprod{\partial_t \vec{m}_k \times \Delta \vec{m}_k}{\boldsymbol{\xi}} + \liprod{\vec{m}_k \times \partial_t \Delta \vec{m}_k}{\boldsymbol{\xi}} \\
    &\qquad + \liprod{\partial_t \vec{m}_k \times \vec{B}_k}{\boldsymbol{\xi}} + \liprod{\vec{m}_k \times \partial_t \vec{B}_k}{\boldsymbol{\xi}}  + \liprod{\partial_t \vec{m}_k \times (\vec{m}_k \times \vec{B}_k)}{\boldsymbol{\xi}} \\
    &\qquad + \liprod{\vec{m}_k \times (\partial_t \vec{m}_k \times \vec{B}_k)}{\boldsymbol{\xi}} + \liprod{\vec{m}_k \times (\vec{m}_k \times \partial_t \vec{B}_k)}{\boldsymbol{\xi}}. 
\end{split}
\end{align}
Setting $\boldsymbol{\xi} = \partial_t \vec{m}_k$ in \eqref{eqn:diff-galerkin-magnetisation} we obtain
\begin{align*}
    \frac{1}{2} \ddt \lnorm{\partial_t \vec{m}_k}{2} + \lnorm{\partial_t \nabla \vec{m}_k}{2}^2 &\lsim |I_1| \cdots + |I_{11}|,
\end{align*}
where each $I_i$ represents the corresponding inner product in (\ref{eqn:diff-galerkin-magnetisation}). To produce these estimates using Young's inequality and the Sobolev embedding theorem. For the inner product $I_1$, we use the Sobolev embedding $\HB^2(\Omega) \hookrightarrow \LB^{\infty}(\Omega)$ to obtain
\begin{align*}
    |I_1| &\lsim \lnorm{\nabla \vec{m}_k}{\infty} \lnorm{\partial_t \vec{v}_k}{2} \lnorm{\partial_t \vec{m}_k}{2} \\
    &\lsim \hnorm{\vec{m}_k}{3}^4 + \lnorm{\partial_t \vec{v}_k}{2}^4 + \lnorm{\partial_t \vec{m}_k}{2}^2 \\
    &\lsim J + J^2.
\end{align*}
The inner product $I_2 = 0$ because $\partial_t \vec{v}_k$ is divergence-free. Similarly, $I_5 = I_7 = I_9 = 0$ due to the orthogonality property of the cross product. The estimates for $I_3$ and $I_4$ follow similarly to $I_1$ so that
\begin{align*}
    |I_3| + |I_4| &\lsim J^2,
\end{align*}
The same is true of $I_6$, however, we need to ensure that we isolate term $\partial_t \Delta \vec{m}_k$ to obtain
\begin{align*}
    |I_6| &\lsim J^2 + \varepsilon \lnorm{\partial_t \Delta \vec{m}_k}{2}^2.
\end{align*}
The remaining estimates also follow similarly to $I_1$ to produce
\begin{equation*}
    |I_8| \lsim J + J^2, \qquad |I_{10}| \lsim J^2, \qquad |I_{11}| \lsim J + J^3.
\end{equation*}
The result is proved by summing up these estimates.
\end{proof}

Due to the nature of the non-linearity in \eqref{eqn:fmhd-equiv e} we also require an estimate on the $\HB^1$-norm of the time derivative of $\vec{m}_k$.

\begin{lemma} \label{lem:time-derivative-estimate-magnetisation-2}
Let $\varepsilon > 0$. For each $k \in \N$ we have
\begin{align*}
    \ddt \lnorm{\partial_t \nabla \vec{m}_k}{2}^2 + \lnorm{\partial_t \Delta \vec{m}_k}{2}^2 &\lsim J^2 + J^3.
\end{align*}
\end{lemma}
\begin{proof}
Setting $\boldsymbol{\xi} = -\partial_t \Delta \vec{m}_k$ in (\ref{eqn:diff-galerkin-magnetisation}), we obtain
\begin{align*}
    \frac{1}{2} \ddt \lnorm{\partial_t \nabla \vec{m}_k}{2}^2 + \lnorm{\partial_t \Delta \vec{m}_k}{2}^2 &\lsim |I_1| + \cdots + |I_{11}|,
\end{align*}
where each $I_i$ represents the corresponding inner product in (\ref{eqn:diff-galerkin-magnetisation}). To produce those estimates, we used Young's inequality and the Sobolev embedding theorem. We estimate $I_1$ using the Sobolev embedding $\HB^2(\Omega) \hookrightarrow \LB^{\infty}(\Omega)$ to obtain
\begin{align*}
    |I_1| &\lsim \lnorm{\nabla \vec{m}_k}{\infty} \lnorm{\partial_t \vec{v}_k}{2} \lnorm{\partial_t \Delta \vec{m}_k}{2} \\
    &\lsim J^2 + \varepsilon \lnorm{\partial_t \Delta \vec{m}_k}{2}^2.
\end{align*}
The estimates for $I_2$, $I_3$ and $I_4$ follow similarly to produce
\begin{align*}
    |I_2| &\lsim J^2 + \varepsilon \lnorm{\partial_t \Delta \vec{m}_k}{2}^2, \\
    |I_3| + |I_4| &\lsim J^3 + \varepsilon \lnorm{\partial_t \Delta \vec{m}_k}{2}^2.
\end{align*}
For $I_5$, we use the additional Sobolev embedding $\HB^1(\Omega) \hookrightarrow \LB^6(\Omega) \hookrightarrow \LB^3(\Omega)$ so that 
\begin{align*}
    |I_5| &\lsim \lnorm{\partial_t \vec{m}_k}{6} \lnorm{\Delta \vec{m}_k}{3} \lnorm{\partial_t \Delta \vec{m}_k}{2} \\
    &\lsim \hnorm{\partial_t \vec{m}_k}{1} \hnorm{\vec{m}_k}{3} \lnorm{\partial_t \Delta \vec{m}_k}{2} \\
    &\lsim \hnorm{\partial_t \vec{m}_k}{1}^4 + \hnorm{\vec{m}_k}{3}^4 + \varepsilon \lnorm{\partial_t \Delta \vec{m}_k}{2}^2 \\
    &\lsim J^2 + \varepsilon \lnorm{\partial_t \Delta \vec{m}_k}{2}^2.
\end{align*}
Due to the orthogonality property of the cross-product, $I_6 = 0$. The next two estimates follow in the same manner as $I_1$ so that we obtain
\begin{align*}
    |I_7| + |I_8| &\lsim J^2 + \varepsilon \lnorm{\partial_t \Delta \vec{m}_k}{2}^2.
\end{align*}
The remaining terms are very similar to $I_7$ and $I_8$ and thus follow in the same manner as $I_1$ as well, giving
\begin{align*}
    |I_9| + |I_{10}| + |I_{11}| &\lsim J^3 + \varepsilon \lnorm{\partial_t \Delta \vec{m}_k}{2}^2.
\end{align*}
Summing up these estimates and absorbing the term $\varepsilon \lnorm{\partial_t \Delta \vec{m}_k}{2}^2$ into the left-hand side gives us the result.
\end{proof}

Throughout Lemmas \ref{lem:first-order-estimate-1}--\ref{lem:time-derivative-estimate-magnetisation-2} we have made liberal use of $J(t)$ which contains $\hnorm{\vec{v}_k}{2}$, $\hnorm{\vec{B}_k}{2}$ and $\hnorm{\vec{m}_k}{3}$. We do not establish estimates for these using the same strategy as before, but rather we pursue a strategy that makes use of the information encoded in the time derivatives. This was a technique used in \cite{chen2011-article} to achieve higher regularity for the solutions. These estimates are summarised in the following lemma.

\begin{lemma} \label{lem:second-order-estimates-velocity-magnetic-field}
For each $k \in \N$, the Galerkin solutions $\vec{v}_k$, $\vec{B}_k$ and $\vec{m}_k$ are bounded such that:
\begin{align*}
    \hnorm{\vec{v}_k}{2}^2 + \hnorm{\vec{B}_k}{2}^2 + \hnorm{\vec{m}_k}{3}^2 &\lsim \lnorm{\partial_t \vec{v}_k}{2}^2 + \hnorm{\vec{v}_k}{1}^2 + \hnorm{\vec{v}_k}{1}^6  \\
    &\qquad + \lnorm{\partial_t \vec{B}_k}{2}^2 + \hnorm{\vec{B}_k}{1}^2 + \hnorm{\vec{B}_k}{1}^6  \\
    &\qquad + \hnorm{\partial_t \vec{m}_k}{1}^2 + \hnorm{\vec{m}_k}{2}^2 + \hnorm{\vec{m}_k}{2}^{10}.
\end{align*}
\end{lemma}
\begin{proof}
Rearrange (\ref{eqn:galerkin-velocity}) such that
\begin{equation} \label{eqn:lem-second-order-estimates-velocity-magnetic-field-rearranged-galerkin-velocity}
    \begin{split}
        \liprod{-\Delta \vec{v}_k}{\boldsymbol{\varphi}} &= -\liprod{\partial_t \vec{v}_k}{\boldsymbol{\varphi}}-\liprod{(\vec{v}_k \cdot \nabla)\vec{v}_k}{\boldsymbol{\varphi}} + \liprod{\curl \vec{B}_k \times \vec{B}_k}{\boldsymbol{\varphi}} + \liprod{(\vec{B}_k \cdot \nabla) \vec{m}_k}{\boldsymbol{\varphi}} \\
        &\qquad - \liprod{(\nabla \vec{m}_k)^{\top} \Delta \vec{m}_k}{\vec{\varphi}}.
    \end{split}
\end{equation}
We make the observation that the eigenfunctions of the Stokes problem \eqref{stokes-problem} $\{\boldsymbol{\varphi}_i\}_{i=1}^k$ are also eigenfunctions of the Stokes operator $A$, see \eqref{eqn:stokes-operator}, since
\begin{equation*}
    A \vec{\varphi}_i = -\PP \Delta \vec{\varphi_i} = -\PP(-\lambda_i \vec{\varphi} + \nabla p_i) = \lambda_i \vec{\varphi}_i,
\end{equation*}
noting that the Helmholtz--Leray decomposition is unique (up to some constant for $p_i$). Consequently, we have (noting that $A$ is linear)
\begin{align*}
    -\Delta \vec{v}_k &= \sum_{i=1}^k a_i(t) (-\Delta \boldsymbol{\varphi}_i) = \sum_{i=1}^k a_i(t) (\lambda_i \boldsymbol{\varphi}_i - \nabla p_i) = A \vec{v}_k  - \sum_{i=1}^k a_i(t) \nabla p_i,
\end{align*}
and so by taking the inner product of both sides with $A\vec{v}_k$ and using the fact that $\div(A \vec{v}_k) = 0$, we obtain
\begin{equation*}
     \liprod{-\Delta \vec{v}_k}{A \vec{v}_k} = \lnorm{A \vec{v}_k}{2}^2.
\end{equation*}
Thus, by setting $\boldsymbol{\varphi} = A \vec{v}_k$ in \eqref{eqn:lem-second-order-estimates-velocity-magnetic-field-rearranged-galerkin-velocity}, using Hölder's inequality, Young's inequality, and Lemma \ref{lem:chen} we obtain
\begin{align*} 
    \lnorm{A \vec{v}_k}{2}^2 &\lsim \lnorm{\partial_t \vec{v}_k}{2}^2 + \lnorm{|\vec{v}_k| \, |\nabla \vec{v}_k|}{2}^2 + \lnorm{|\vec{B}_k| \, |\nabla \vec{B}_k|}{2}^2 + \lnorm{|\vec{B}_k| \, |\nabla \vec{m}_k|}{2}^2  \\
    &\qquad + \lnorm{|\nabla \vec{m}_k| \, |\Delta \vec{m}_k|}{2}^2  \\
    &\lsim \lnorm{\partial_t \vec{v}_k}{2}^2 + \hnorm{\vec{v}_k}{1}^3 \hnorm{\vec{v}_k}{2} + \hnorm{\vec{B}_k}{1}^3 \hnorm{\vec{B}_k}{2}  \\
    &\qquad + \hnorm{\vec{B}_k}{1}^2 \hnorm{\vec{m}_k}{1} \hnorm{\vec{m}_k}{2} + \lnorm{\nabla \vec{m}_k}{6}^2 \lnorm{\Delta \vec{m}_k}{3}^2.
\end{align*}
Now, estimating the last term on the right-hand side requires additional effort by way of Lemma \ref{lem:carbou}. We have for sufficiently small $\varepsilon > 0$, with the Sobolev embedding $\HB^1(\Omega) \hookrightarrow \LB^6(\Omega)$,
\begin{align*}
    \lnorm{\nabla \vec{m}_k}{6}^2 \lnorm{\Delta \vec{m}_k}{3}^2 &\lsim \hnorm{\vec{m}_k}{2}^2 \left(\hnorm{\vec{m}_k}{2}^2 + \hnorm{\vec{m}_k}{2} \lnorm{\nabla \Delta \vec{m}_k}{2}\right) \\
    &\lsim \hnorm{\vec{m}_k}{2}^4 + \hnorm{\vec{m}_k}{2}^6 + \varepsilon \lnorm{\nabla \Delta \vec{m}_k}{2}^2.
\end{align*}
Continuing with our analysis, we now obtain
\begin{align} \label{eqn:Avk-estimate}
    \lnorm{A \vec{v}_k}{2}^2 &\lsim \lnorm{\partial_t \vec{v}_k}{2}^2 + \hnorm{\vec{v}_k}{1}^6  + \hnorm{\vec{B}_k}{1}^6 + \hnorm{\vec{m}_k}{1}^6 + \hnorm{\vec{m}_k}{2}^2  \nonumber \\
    &\qquad \hnorm{\vec{m}_k}{2}^4 + \hnorm{\vec{m}_k}{2}^6 + \varepsilon \hnorm{\vec{v}_k}{2}^2 + \varepsilon \hnorm{\vec{B}_k}{2}^2 + \varepsilon \lnorm{\nabla \Delta \vec{m}_k}{2}^2 \nonumber \\
    &\lsim \lnorm{\partial_t \vec{v}_k}{2}^2 + \hnorm{\vec{v}_k}{1}^6  + \hnorm{\vec{B}_k}{1}^6 + \hnorm{\vec{m}_k}{2}^6 + \hnorm{\vec{v}_k}{1}^2  \nonumber \\
    &\qquad + \hnorm{\vec{B}_k}{1}^2 + \hnorm{\vec{m}_k}{2}^2 + \hnorm{\vec{m}_k}{2}^4 + \hnorm{\vec{m}_k}{2}^6 \nonumber \\
    &\qquad +  \varepsilon \lnorm{\nabla^2 \vec{v}_k}{2}^2 + \varepsilon \lnorm{\nabla^2 \vec{B}_k}{2}^2 + \varepsilon \lnorm{\nabla \Delta \vec{m}_k}{2}^2.
\end{align}
Next, we rearrange (\ref{eqn:galerkin-magnetic-field}) to obtain
\begin{equation} \label{eqn:lem-second-order-estimates-velocity-magnetic-field-rearranged-galerkin-magnetic-field}
    \liprod{\curl^2 \vec{B}_k}{\boldsymbol{\omega}} =-\liprod{\partial_t \vec{B}_k}{\boldsymbol{\boldsymbol{\omega}}} +  \liprod{\curl(\vec{v}_k \times \vec{B}_k)}{\boldsymbol{\omega}}.
\end{equation}
Setting $\boldsymbol{\omega} = \curl^2 \vec{B}_k$, applying \eqref{eqn:curl-cross-product}, and using the same procedure as above, we obtain
\begin{align} \label{eqn:curl2Bk-estimate}
    \lnorm{\curl^2 \vec{B}_k}{2}^2 &\lsim \lnorm{\partial_t \vec{B}_k}{2}^2 + \lnorm{|\nabla \vec{v}_k| \, |\vec{B}_k|}{2}^2 + \lnorm{|\vec{v}_k| \, |\nabla \vec{B}_k|}{2}^2 \nonumber \\
    &\lsim \lnorm{\partial_t \vec{B}_k}{2}^2 + \hnorm{\vec{B}_k}{1}^2 \hnorm{\vec{v}_k}{1} \hnorm{\vec{v}_k}{2} + \hnorm{\vec{v}_k}{1}^2 \hnorm{\vec{B}_k}{1} \hnorm{\vec{B}_k}{2}  \nonumber \\
    &\lsim \lnorm{\partial_t \vec{B}_k}{2}^2 + \hnorm{\vec{B}_k}{1}^6 + \hnorm{\vec{v}_k}{1}^6 + \varepsilon \hnorm{\vec{v}_k}{2}^2 + \varepsilon \hnorm{\vec{B}_k}{2}^2. \nonumber \\
    &\lsim \lnorm{\partial_t \vec{B}_k}{2}^2 + \hnorm{\vec{B}_k}{1}^6 + \hnorm{\vec{v}_k}{1}^6 + \hnorm{\vec{v}_k}{1}^2 + \hnorm{\vec{B}_k}{1}^2 \nonumber \\
    &\qquad + \varepsilon \lnorm{\nabla^2 \vec{v}_k}{2}^2 + \varepsilon \lnorm{\nabla^2 \vec{B}_k}{2}^2. 
\end{align}
Lastly, we rearrange \eqref{eqn:galerkin-magnetisation} so that we have
\begin{align*}
    -\liprod{\Delta \vec{m}_k}{\boldsymbol{\xi}} &= -\liprod{\partial_t \vec{m}_k}{\boldsymbol{\xi}} -\liprod{(\vec{v}_k \cdot \nabla)\vec{m}_k}{\boldsymbol{\xi}} + \liprod{|\nabla \vec{m}_k|^2 \vec{m}_k}{\boldsymbol{\xi}} + \liprod{\vec{m}_k \times \Delta \vec{m}_k}{\boldsymbol{\xi}} \\
    &\qquad + \liprod{\vec{m}_k \times \vec{B}_k}{\boldsymbol{\xi}}  + \liprod{\vec{m}_k \times (\vec{m}_k \times \vec{B}_k)}{\boldsymbol{\xi}}.
\end{align*}
Setting $\vec{\xi} = \Delta^2 \vec{m}_k$ and integrating by parts we obtain
\begin{align*}
    \lnorm{\nabla \Delta \vec{m}_k}{2}^2 &= \liprod{\partial_t \nabla \vec{m}_k}{\nabla \Delta \vec{m}_k} + \liprod{\nabla \Big((\vec{v}_k \cdot \nabla)\vec{m}_k\Big)}{\nabla \Delta \vec{m}_k}  - \liprod{\nabla \Big(|\nabla \vec{m}_k|^2 \vec{m}_k\Big)}{\nabla \Delta \vec{m}_k} \\
    &\qquad - \liprod{\nabla\Big(\vec{m}_k \times \Delta \vec{m}_k\Big)}{\nabla \Delta \vec{m}_k} - \liprod{\nabla(\vec{m}_k \times \vec{B}_k\Big)}{\nabla \Delta \vec{m}_k} \\
    &\qquad - \liprod{\nabla \Big(\vec{m}_k \times (\vec{m}_k \times \vec{B}_k)\Big)}{\nabla \Delta \vec{m}_k} \\
    &\lsim |I_1| + \cdots + |I_6|,
\end{align*}
where each $I_i$ represents the corresponding inner product in the equation above. Previously, we treated $I_2$--$I_6$ liberally in Lemma \ref{lem:second-order-estimate-magnetisation} but here we have to be more careful with our estimates. As usual, we produce these estimates by Young's inequality and the Sobolev embedding theorem. For sufficiently small $\varepsilon > 0$, we estimate $I_1$ to be
\begin{align*}
    |I_1| &\lsim \lnorm{\partial_t \nabla \vec{m}_k}{2}^2 + \varepsilon \lnorm{\nabla \Delta \vec{m}_k}{2}^2.
\end{align*}
To obtain the estimate for $I_2$, we use the Gagliardo-Nirenberg inequality (Theorem \ref{thm:gagliardo-nirenberg}), the Sobolev embedding $\HB^1(\Omega) \hookrightarrow \LB^6(\Omega) \hookrightarrow \LB^3(\Omega)$ and Lemma \ref{lem:carbou} so that we have
\begin{align*}
    |I_2| &\lsim \lnorm{\nabla \vec{m}_k}{6} \lnorm{\nabla \vec{v}_k}{3} \lnorm{\nabla \Delta \vec{m}_k}{2} + \lnorm{\vec{v}_k}{6} \lnorm{\nabla^2 \vec{m}_k}{3} \lnorm{\nabla \Delta \vec{m}_k}{2} \\
    &\lsim \hnorm{\vec{m}_k}{2} \hnorm{\vec{v}_k}{1}^{1/2} \hnorm{\vec{v}_k}{2}^{1/2} \lnorm{\nabla \Delta \vec{m}_k}{2} \\
    &\qquad + \hnorm{\vec{v}_k}{1} \left(\hnorm{\vec{m}_k}{2} + \hnorm{\vec{m}_k}{2}^{1/2} \lnorm{\nabla \Delta \vec{m}_k}{2}^{1/2}\right) \lnorm{\nabla \Delta \vec{m}_k}{2} \\
    &\lsim \hnorm{\vec{m}_k}{2}^6 + \hnorm{\vec{v}_k}{1}^6 + \hnorm{\vec{v}_k}{1}^2 + \varepsilon \lnorm{\nabla^2 \vec{v}_k}{2}^2 + \varepsilon \lnorm{\nabla \Delta \vec{m}_k}{2}^2 \\
    &\qquad + \hnorm{\vec{v}_k}{1}^4 + \hnorm{\vec{m}_k}{2}^4 + \hnorm{\vec{v}_k}{1}^6 + \hnorm{\vec{m}_k}{2}^6 + \varepsilon \lnorm{\nabla \Delta \vec{m}_k}{2}^2 \\
    &\lsim \hnorm{\vec{v}_k}{1}^2 + \hnorm{\vec{v}_k}{1}^4 + \hnorm{\vec{v}_k}{1}^6 + \hnorm{\vec{m}_k}{2}^4 + \hnorm{\vec{m}_k}{2}^6 \\
    &\qquad + \varepsilon \lnorm{\nabla^2 \vec{v}_k}{2}^2 + \varepsilon \lnorm{\nabla \Delta \vec{m}_k}{2}^2.
\end{align*}
We estimate $I_3$ as we did $I_2$ but with the Sobolev embedding $\HB^2(\Omega) \hookrightarrow \LB^{\infty}(\Omega)$ in addition, to produce
\begin{align*}
    |I_3| &\lsim \lnorm{\vec{m}_k}{\infty} \lnorm{\nabla \vec{m}_k}{6} \lnorm{\nabla^2 \vec{m}_k}{3} \lnorm{\nabla \Delta \vec{m}_k}{2} + \lnorm{\nabla \vec{m}_k}{6}^3 \lnorm{\nabla \Delta \vec{m}_k}{2} \\
    &\lsim \hnorm{\vec{m}_k}{2}^2 \left(\hnorm{\vec{m}_k}{2} + \hnorm{\vec{m}_k}{2}^{1/2} \lnorm{\nabla \Delta \vec{m}_k}{2}^{1/2}\right) \lnorm{\nabla \Delta \vec{m}_k}{2} \\
    &\qquad + \hnorm{\vec{m}_k}{2}^3 \lnorm{\nabla \Delta \vec{m}_k}{2} \\
    &\lsim \hnorm{\vec{m}_k}{2}^6 + \hnorm{\vec{m}_k}{2}^{10} + \varepsilon \lnorm{\nabla \Delta \vec{m}_k}{2}^2.
\end{align*}
The estimate for $I_4$ is simplified due to the orthogonality property of the cross-product. We bound it in a similar manner to $I_2$ so that we have
\begin{align*}
    |I_4| &\lsim \hnorm{\vec{m}_k}{2}^4 + \hnorm{\vec{m}_k}{2}^6 + \varepsilon \lnorm{\nabla \Delta \vec{m}_k}{2}^2.
\end{align*}
The remaining terms can be estimated using the same embeddings as $I_3$ to obtain
\begin{align*}
    |I_5| &\lsim \hnorm{\vec{B}_k}{1}^4 + \hnorm{\vec{m}_k}{2}^4 + \varepsilon \lnorm{\nabla \Delta \vec{m}_k}{2}^2, \\
    |I_6| &\lsim \hnorm{\vec{B}_k}{1}^6 + \hnorm{\vec{m}_k}{2}^6 + \varepsilon \lnorm{\nabla \Delta \vec{m}_k}{2}^2. 
\end{align*}
Summing up these estimates, we have
\begin{align} \label{eqn:nabla-delta-mk-estimate}
    \lnorm{\nabla \Delta \vec{m}_k}{2}^2 &\lsim \hnorm{\vec{v}_k}{1}^2 + \hnorm{\vec{v}_k}{1}^4 + \hnorm{\vec{v}_k}{1}^6 + \hnorm{\vec{B}_k}{1}^4 + \hnorm{\vec{B}_k}{1}^6 \nonumber \\
    &\qquad + \hnorm{\partial_t \vec{m}_k}{1}^2 + \hnorm{\vec{m}_k}{2}^4 + \hnorm{\vec{m}_k}{2}^6 + \hnorm{\vec{m}_k}{2}^{10} \nonumber \\
    &\qquad + \varepsilon \lnorm{\nabla^2 \vec{v}_k}{2}^2 + \varepsilon \lnorm{\nabla \Delta \vec{m}_k}{2}^2.
\end{align}
By \cite[Proposition 4.7]{constantin1988-book}, that is
\begin{equation*}
    \lnorm{\nabla^2 \vec{u}}{2} \lsim \lnorm{A \vec{u}}{2}, \qquad \forall \vec{u} \in \HB^2(\Omega),
\end{equation*}
and Lemma \ref{lem:curl-estimate}, and the estimates \eqref{eqn:Avk-estimate}, \eqref{eqn:curl2Bk-estimate} and \eqref{eqn:nabla-delta-mk-estimate}, we conclude that for sufficiently small $\varepsilon$, 
\begin{align*}
    &\lnorm{\nabla^2 \vec{v}_k}{2}^2 + \lnorm{\nabla^2 \vec{B}_k}{2}^2 + \lnorm{\nabla \Delta \vec{m}_k}{2}^2 \\
    &\qquad \lsim \lnorm{\partial_t \vec{v}_k}{2}^2 + \hnorm{\vec{v}_k}{1}^2 + \hnorm{\vec{v}_k}{1}^4 + \hnorm{\vec{v}_k}{1}^6  \\
    &\qquad \qquad + \lnorm{\partial_t \vec{B}_k}{2}^2 + \hnorm{\vec{B}_k}{1}^2 + \hnorm{\vec{B}_k}{1}^4 + \hnorm{\vec{B}_k}{1}^6  \\
    &\qquad \qquad + \hnorm{\partial_t \vec{m}_k}{1}^2 + \hnorm{\vec{m}_k}{2}^2 + \hnorm{\vec{m}_k}{2}^4 + \hnorm{\vec{m}_k}{2}^6 + \hnorm{\vec{m}_k}{2}^{10}.
\end{align*}
The result is proved by noting that
\begin{align*}
    \hnorm{\vec{v}_k}{2}^2 &\lsim \hnorm{\vec{v}_k}{1}^2 + \lnorm{\nabla^2 \vec{v}_k}{2}^2, \\
    \hnorm{\vec{B}_k}{2}^2 &\lsim \hnorm{\vec{B}_k}{1}^2 + \lnorm{\nabla^2 \vec{B}_k}{2}^2, \\
    \hnorm{\vec{m}_k}{3}^2 &\lsim \hnorm{\vec{m}_k}{2}^2 + \lnorm{\nabla \Delta \vec{m}_k}{2}^2,
\end{align*}
where the latter inequality is given by \eqref{eqn:lem-carbou-eqn-2} in Lemma \ref{lem:carbou}.
\end{proof}

Before we apply the Gronwall Lemma in Proposition \ref{prop:combined-estimates}, we need to ensure that the time derivatives of $\vec{v}_k$, $\vec{B}_k$ and $\vec{m}_k$ are appropriately bounded by the initial data.

\begin{proposition} \label{prop:limit-time-derivative-initial-data}
For each $k \in \N$, we have
\begin{align*}
    \lnorm{\partial_t \vec{v}_k(0)}{2}^2 + \lnorm{\partial_t \vec{B}_k(0)}{2}^2 + \hnorm{\partial_t \vec{m}_k(0)}{1}^2 \lsim 1.
\end{align*}
\end{proposition}
\begin{proof}
From \eqref{eqn:galerkin-ode} we know that the limit as $t \to 0^+$ of $\partial_t \vec{v}_k$, $\partial_t \vec{B}_k$ and $\partial_t \vec{m}_k$ exist. Now, we seek to make use of the converging sequence $(\vec{v}_0^k, \vec{B}_0^k) \to (\vec{v}_0, \vec{B}_0)$ and $\vec{m}_0^k \to \vec{m}_0$ in $\HB^2(\Omega)$ and $\HB^3(\Omega)$, respectively, to bound the values of $\partial_t \vec{v}_k$, $\partial_t \vec{B}_k$ and $\partial_t \vec{m}_k$ at $t=0$. Rearranging \eqref{eqn:lem-second-order-estimates-velocity-magnetic-field-rearranged-galerkin-velocity} and setting $t=0$ we obtain
\begin{align*}
    \liprod{\partial_t \vec{v}_k(0)}{\vec{\varphi}} &= \liprod{\Delta \vec{v}_k^0}{\vec{\varphi}} - \liprod{(\vec{v}_k^0 \cdot \nabla) \vec{v}_k^0}{\vec{\varphi}} + \liprod{\curl \vec{B}_k^0 \times \vec{B}_k^0}{\vec{\varphi}} + \liprod{(\vec{B}_k^0 \cdot \nabla)\vec{m}_k^0}{\vec{\varphi}} \\
    &\qquad - \liprod{(\nabla \vec{m}_k^0)^{\top} \Delta \vec{m}_k^0}{\vec{\varphi}}.
\end{align*}
Setting $\vec{\varphi} = \partial_t \vec{v}_k(0)$, but omitting the initial data notation for the sake of simplicity, and using the Sobolev embedding $\HB^2(\Omega) \hookrightarrow \LB^{\infty}(\Omega)$ we have
\begin{align*}
    \lnorm{\partial_t \vec{v}_k}{2}^2 &\lsim \lnorm{\Delta \vec{v}_k}{2} \lnorm{\partial_t \vec{v}_k}{2} + \lnorm{\vec{v}_k}{\infty} \lnorm{\nabla \vec{v}_k}{2} \lnorm{\partial_t \vec{v}_k}{2} \\
    &\qquad + \lnorm{\vec{B}_k}{\infty} \lnorm{\curl \vec{B}_k}{2} \lnorm{\partial_t \vec{v}_k}{2} + \lnorm{\vec{B}_k}{\infty} \lnorm{\nabla \vec{m}_k}{2} \lnorm{\partial_t \vec{v}_k}{2} \\
    &\qquad + \lnorm{\nabla \vec{m}_k}{\infty} \lnorm{\Delta \vec{m}_k}{2} \lnorm{\partial_t \vec{v}_k}{2} \\
    &\lsim \hnorm{\vec{v}_k}{2} \lnorm{\partial_t \vec{v}_k}{2} + \hnorm{\vec{v}_k}{2}^2 \lnorm{\partial_t \vec{v}_k}{2} + \hnorm{\vec{B}_k}{2}^2 \lnorm{\partial_t \vec{v}_k}{2} \\
    &\qquad + \hnorm{\vec{B}_k}{2} \hnorm{\vec{m}_k}{1} \lnorm{\partial_t \vec{v}_k}{2} + \hnorm{\vec{m}_k}{3}^2 \lnorm{\partial_t \vec{v}_k}{2}.
\end{align*}
Using Young's inequality, we can bring $\partial_t \vec{v}$ over to the left-hand side so that
\begin{align*}
    \lnorm{\partial_t \vec{v}_k(0)}{2}^2 &\lsim \hnorm{\vec{v}_k}{2}^2 + \hnorm{\vec{v}_k}{2}^4 + \hnorm{\vec{B}_k}{2}^4 + \hnorm{\vec{m}_k}{3}^4 \\
    &= \hnorm{\vec{v}_k^0}{2}^2 + \hnorm{\vec{v}_k^0}{2}^4 + \hnorm{\vec{B}_k^0}{2}^4 + \hnorm{\vec{m}_k^0}{3}^4 \\
    &\lsim \hnorm{\vec{v}_0}{2}^2 + \hnorm{\vec{v}_0}{2}^4 + \hnorm{\vec{B}_0}{2}^4 + \hnorm{\vec{m}_0}{3}^4,
\end{align*}
since $(\vec{v}_0^k, \vec{B}_0^k) \to (\vec{v}_0, \vec{B}_0)$ and $\vec{m}_0^k \to \vec{m}_0$ in $\HB^2(\Omega)$ and $\HB^3(\Omega)$, respectively. Similarly, rearranging \eqref{eqn:lem-second-order-estimates-velocity-magnetic-field-rearranged-galerkin-magnetic-field} and using the same procedure, along with Lemma \ref{lem:curl-vxb}, we obtain 
\begin{equation*}
    \lnorm{\partial_t \vec{B}_k(0)}{2}^2 \lsim \hnorm{\vec{v}_0}{2}^4 + \hnorm{\vec{B}_0}{2}^2 + \hnorm{\vec{B}_0}{2}^4.
\end{equation*}
Repeating this process, by first rearranging \eqref{eqn:galerkin-magnetisation} with $t=0$, and then using the additional Sobolev embedding $\HB^1(\Omega) \hookrightarrow \LB^6(\Omega) \hookrightarrow \LB^4(\Omega)$, we have
\begin{equation*}
    \lnorm{\partial_t \vec{m}_k(0)}{2}^2 \lsim \hnorm{\vec{v}_0}{2}^4 + \hnorm{\vec{B}_0}{2}^4 + \hnorm{\vec{m}_0}{2}^2 + \hnorm{\vec{m}_0}{2}^8.
\end{equation*}
However, we still require additional regularity on the time derivative for $\vec{m}_k$. Setting $t=0$ and $\vec{\xi} = -\partial_t \Delta \vec{m}_k(0)$ in \eqref{eqn:galerkin-magnetisation} and integrating by parts we have
\begin{align*}
    \lnorm{\partial_t \nabla \vec{m}_k}{2}^2 &= \liprod{\nabla \Delta \vec{m}_k}{\partial_t \nabla \vec{m}_k} - \liprod{\nabla [(\vec{v}_k \cdot \nabla)\vec{m}_k]}{\partial_t \nabla \vec{m}_k} + \liprod{\nabla (|\nabla \vec{m}_k|^2 \vec{m}_k)}{\partial_t \nabla \vec{m}_k}  \\
    &\qquad + \liprod{\nabla(\vec{m}_k \times \Delta \vec{m}_k)}{\partial_t \nabla \vec{m}_k} + \liprod{\nabla(\vec{m}_k \times \vec{B}_k)}{\partial_t \nabla \vec{m}_k} \\
    &\qquad + \liprod{\nabla[\vec{m}_k \times (\vec{m}_k \times \vec{B}_k)]}{\partial_t \nabla \vec{m}_k}.
\end{align*}
We thus obtain
\begin{align*}
    \lnorm{\partial_t \nabla \vec{m}_k(0)}{2}^2 &\lsim \lnorm{\nabla \Delta \vec{m}_k}{2} \lnorm{\partial_t \nabla \vec{m}_k}{2} + \lnorm{\nabla \vec{m}_k}{\infty} \lnorm{\nabla \vec{v}_k}{2} \lnorm{\partial_t \nabla \vec{m}_k}{2}\\
    &\qquad + \lnorm{\vec{v}_k}{\infty} \lnorm{\nabla^2 \vec{m}_k}{2} \lnorm{\partial_t \nabla \vec{m}_k}{2} \\
    &\qquad + \lnorm{\vec{m}_k}{\infty} \lnorm{\nabla \vec{m}_k}{\infty} \lnorm{\nabla^2 \vec{m}_k}{2} \lnorm{\partial_t \nabla \vec{m}_k}{2} \\
    &\qquad + \lnorm{\nabla \vec{m}_k}{\infty}^2 \lnorm{\nabla \vec{m}_k}{2} \lnorm{\partial_t \nabla \vec{m}_k}{2} \\
    &\qquad + \lnorm{\nabla \vec{m}_k}{\infty} \lnorm{\Delta \vec{m}_k}{2} \lnorm{\partial_t \nabla \vec{m}_k}{2} \\
    &\qquad + \lnorm{\vec{m}_k}{\infty} \lnorm{\nabla \Delta \vec{m}_k}{2} \lnorm{\partial_t \nabla \vec{m}_k}{2} \\
    &\qquad + \lnorm{\vec{B}_k}{\infty} \lnorm{\nabla \vec{m}_k}{2} \lnorm{\partial_t \nabla \vec{m}_k}{2} \\
    &\qquad + \lnorm{\vec{m}_k}{\infty} \lnorm{\nabla \vec{B}_k}{2} \lnorm{\partial_t \nabla \vec{m}_k}{2} \\
    &\qquad + \lnorm{\vec{B}_k}{\infty} \lnorm{\vec{m}_k}{\infty} \lnorm{\nabla \vec{m}_k}{2} \lnorm{\partial_t \nabla \vec{m}_k}{2} \\
    &\qquad + \lnorm{\vec{m}_k}{\infty}^2 \lnorm{\nabla \vec{B}_k}{2} \lnorm{\partial_t \nabla \vec{m}_k}{2} \\
    &\lsim \hnorm{\vec{v}_k}{2}^4 + \hnorm{\vec{B}_k}{2}^4 + \hnorm{\vec{m}_k}{2}^4 + \hnorm{\vec{m}_k}{2}^8 + \hnorm{\vec{m}_k}{3}^2 \\
    &\qquad + \hnorm{\vec{m}_k}{3}^4 + \hnorm{\vec{m}_k}{3}^6.
\end{align*}
Therefore,
\begin{align*}
    \lnorm{\partial_t \nabla \vec{m}_k(0)}{2}^2 &\lsim \hnorm{\vec{v}_0}{2}^4 + \hnorm{\vec{B}_0}{2}^4 + \hnorm{\vec{m}_0}{2}^4 + \hnorm{\vec{m}_0}{2}^8 + \hnorm{\vec{m}_0}{3}^2 \\
    &\qquad + \hnorm{\vec{m}_0}{3}^6,
\end{align*}
thus completing the proof.
\end{proof}

We are now in a position to combine all our precursory inequalities to finally establish estimates for $\vec{v}_k$, $\vec{B}_k$, and $\vec{m}_k$.

\begin{proposition} \label{prop:combined-estimates}
There exists $T^* > 0$ satisfying $T^* \leq T$ such that for any $t \in (0,T^*]$
\begin{equation*}
    \sup_{0 \leq s \leq t} \left( \hnorm{\vec{v}_k(s)}{2}^2 + \hnorm{\vec{B}_k(s)}{2}^2 + \hnorm{\vec{m}_k(s)}{3}^2\right) \lsim 1,
\end{equation*}
\begin{equation*}
    \sup_{0 \leq s \leq t} \left(\lnorm{\partial_s \vec{v}_k(s)}{2}^2 + \lnorm{\partial_s \vec{B}_k(s)}{2}^2 + \hnorm{\partial_s \vec{m}_k(s)}{1}^2\right) \lsim 1
\end{equation*}
and
\begin{equation*}
    \intOt{\left(\lnorm{\partial_s \nabla \vec{v}_k(s)}{2}^2 + \lnorm{\partial_s \nabla \vec{B}_k(s)}{2}^2 + \lnorm{\partial_s \Delta \vec{m}_k(s)}{2}^2\right)} \lsim 1.
\end{equation*}
\end{proposition}
\begin{proof}
Combining the results from Lemmas \ref{lem:first-order-estimate-1}--\ref{lem:second-order-estimate-magnetisation}, using Lemma \ref{lem:curl-estimate} and absorbing the remaining $\varepsilon$-terms into the left-hand side, we obtain
\begin{align*}
    &\ddt \Big(\hnorm{\vec{v}_k}{1}^2 + \lnorm{\vec{B}_k}{2}^2 + \lnorm{\curl \vec{B}_k}{2}^2 + \lnorm{\vec{m}_k}{2}^2 + \lnorm{\Delta \vec{m}_k}{2}^2 \\
    &\qquad + \lnorm{\partial_t \vec{v}_k}{2}^2 + \lnorm{\partial_t \vec{B}_k}{2}^2 + \lnorm{\partial_t \vec{m}_k}{2}^2 + \lnorm{\partial_t \nabla \vec{m}_k}{2}^2\Big) \\
    &\qquad \qquad + \left(\lnorm{\partial_t \nabla \vec{v}_k}{2}^2 + \lnorm{\partial_t \nabla \vec{B}_k}{2}^2 + \lnorm{\partial_t \Delta \vec{m}_k}{2}^2 + \lnorm{\nabla \Delta \vec{m}_k}{2}^2\right) \\
    &\qquad \qquad \qquad \lsim J + J^3.
\end{align*}
Integrating over $[0,t]$, using Proposition \ref{prop:limit-time-derivative-initial-data} and using Lemmas \ref{lem:curl-estimate} and \ref{lem:carbou}, we have
\begin{align*}
    &\hnorm{\vec{v}_k}{1}^2 + \hnorm{\vec{B}_k}{1}^2 + \hnorm{\vec{m}_k}{2}^2  + \lnorm{\partial_t \vec{v}_k}{2}^2 + \lnorm{\partial_t \vec{B}_k}{2}^2 + \hnorm{\partial_t \vec{m}_k}{1}^2 \\
    &\qquad + \intOt{\left(\lnorm{\partial_s \nabla \vec{v}_k(s)}{2}^2 + \lnorm{\partial_s \nabla \vec{B}_k(s)}{2}^2 + \lnorm{\partial_s \Delta \vec{m}_k(s)}{2}^2 + \lnorm{\nabla \Delta \vec{m}_k(s)}{2}^2\right)} \\
    &\qquad \qquad \lsim 1 + \intOt{J(s)  + J^3(s)}.
\end{align*}
Let
\begin{equation*}
    E(t) := \hnorm{\vec{v}_k}{1}^2 + \hnorm{\vec{B}_k}{1}^2 + \hnorm{\vec{m}_k}{2}^2  + \lnorm{\partial_t \vec{v}_k}{2}^2 + \lnorm{\partial_t \vec{B}_k}{2}^2 + \hnorm{\partial_t \vec{m}_k}{1}^2,
\end{equation*}
and note that by Lemma \ref{lem:second-order-estimates-velocity-magnetic-field}
\begin{align*}
    \hnorm{\vec{v}_k}{2}^2 + \hnorm{\vec{B}_k}{2}^2 + \hnorm{\vec{m}_k}{3}^2 \lsim E + E^5;
\end{align*}
in other words,
\begin{equation*}
    J \lsim E + E^5.
\end{equation*}
Hence, by repeated use of Young's inequality to increase the exponent, we find that
\begin{align*}
    &E(t) + \intOt{\left(\lnorm{\partial_s \nabla \vec{v}_k(s)}{2}^2 + \lnorm{\partial_s \nabla \vec{B}_k(s)}{2}^2 + \lnorm{\partial_s \Delta \vec{m}_k(s)}{2}^2 + \lnorm{\nabla \Delta \vec{m}_k(s)}{2}^2\right)}\\
    &\qquad \lsim 1 + \intOt{E^{15}(s)}.
\end{align*}
By the Generalised Gronwall Lemma (Theorem \ref{generalised-gronwall-lemma}), the result is proved.
\end{proof}

\section{Proof of Theorem \ref{thm:existence}} \label{sec:existence}

We now take $\vec{v}_k$, $\vec{B}_k$ and $\vec{m}_k$ and pass to the limit to show that their limits are solutions in the sense of Definition \ref{def:weak-solution}.

Proposition \ref{prop:combined-estimates} and the Banach-Alaoglu theorem imply the existence of a subsequence of Galerkin solutions $\{(\vec{v}_k, \vec{B}_k, \vec{m}_k)\}$, still denoted by $\{(\vec{v}_k, \vec{B}_k, \vec{m}_k)\}$, such that
\begin{equation} \label{eqn:weak-convergence}
    \begin{cases}
        \begin{aligned}
            \vec{v}_k &\weakto \vec{v} &&\quad \text{weakly in } L^2(0,T^*; \HB^2(\Omega)), \\
            \vec{v}_k &\wkstarto \vec{v} &&\quad \text{$\text{weakly}^{\star}$ in } L^{\infty}(0,T^*; \HB^2(\Omega)), \\
            \partial_t \vec{v}_k &\wkstarto \partial_t \vec{v} &&\quad \text{$\text{weakly}^{\star}$ in } L^{\infty}(0,T^*; \LB^2(\Omega)), \\
            \partial_t \vec{v}_k &\weakto \partial_t \vec{v} &&\quad \text{weakly in } L^2(0,T^*; \HB^1(\Omega)), \\
            &&& \\
            \vec{B}_k &\weakto \vec{B} &&\quad\text{weakly in } L^2(0,T^*; \HB^2(\Omega)), \\
            \vec{B}_k &\wkstarto \vec{B} &&\quad\text{$\text{weakly}^{\star}$ in } L^{\infty}(0,T^*; \HB^2(\Omega)), \\
            \partial_t \vec{B}_k &\wkstarto \partial_t \vec{B} &&\quad\text{$\text{weakly}^{\star}$ in } L^{\infty}(0,T^*; \LB^2(\Omega)), \\
            \partial_t \vec{B}_k &\weakto \partial_t \vec{B} &&\quad\text{weakly in } L^2(0,T^*; \HB^1(\Omega)), \\
            &&& \\
            \vec{m}_k &\weakto \vec{m} &&\quad\text{weakly in } L^2(0,T^*; \HB^3(\Omega)), \\
            \vec{m}_k &\wkstarto \vec{m} &&\quad\text{$\text{weakly}^{\star}$ in } L^{\infty}(0,T^*; \HB^3(\Omega)), \\
            \partial_t \vec{m}_k &\wkstarto \partial_t \vec{m} &&\quad\text{$\text{weakly}^{\star}$ in } L^{\infty}(0,T^*; \HB^1(\Omega)), \\
            \partial_t \vec{m}_k &\weakto \partial_t \vec{m} &&\quad\text{weakly in } L^2(0,T^*; \HB^2(\Omega)). \\
        \end{aligned}
    \end{cases}
\end{equation}
By Aubin's Lemma (Theorem \ref{aubins-lemma}), a further subsequence then satisfies
\begin{equation} \label{eqn:strong-convergence}
    \begin{cases}
        \begin{aligned}
            \vec{v}_k &\to \vec{v} &&\quad \text{strongly in } L^2(0,T^*; \HB^1(\Omega)), \\
            \vec{B}_k &\to \vec{B} &&\quad \text{strongly in } L^2(0,T^*; \HB^1(\Omega)), \\
            \vec{m}_k &\to \vec{m} &&\quad \text{strongly in } L^2(0,T^*; \HB^2(\Omega)).
        \end{aligned}
    \end{cases}
\end{equation}

Having established strong convergence, we now show that the non-linear terms in \eqref{eqn:galerkin-velocity}--\eqref{eqn:galerkin-magnetisation} converge to their limiting counterparts. The next three lemmas are concerned with this.

\begin{lemma} \label{lem:convergence-velocity}
Assume that $k \to \infty$. The non-linear terms in \eqref{eqn:galerkin-velocity} converge to their respective limits
\begin{align}
    (\vec{v}_k \cdot \nabla)\vec{v}_k &\to (\vec{v} \cdot \nabla)\vec{v}, \label{eqn:convergence-velocity-1} \\[2ex]
    \curl \vec{B}_k \times \vec{B}_k &\to \curl \vec{B} \times \vec{B}, \label{eqn:convergence-velocity-2} \\[2ex]
    (\vec{B}_k \cdot \nabla)\vec{m}_k &\to (\vec{B} \cdot \nabla)\vec{m}, \label{eqn:convergence-velocity-3} \\[2ex]
    (\nabla \vec{m}_k)^{\top} \Delta \vec{m}_k &\to (\nabla \vec{m})^{\top} \Delta \vec{m}. \label{eqn:convergence-velocity-4}
\end{align}
in the sense of $L^2(0,t; \LB^2(\Omega))$, for all $t \in [0,T^*]$.
\end{lemma}
\begin{proof}
Using \eqref{eqn:weak-convergence} and \eqref{eqn:strong-convergence}, and the Sobolev embeddings $\HB^1(\Omega) \hookrightarrow \LB^6(\Omega) \hookrightarrow \LB^4(\Omega)$ and~$\HB^2(\Omega) \hookrightarrow \LB^{\infty}(\Omega)$, we obtain the limit \eqref{eqn:convergence-velocity-1} since
\begin{align*}
    &\intOt{\lnorm{(\vec{v}_k \cdot \nabla)\vec{v}_k - (\vec{v} \cdot \nabla)\vec{v}}{2}^2} \\
    &\qquad \lsim \intOt{\lnorm{(\vec{v}_k \cdot \nabla)(\vec{v}_k - \vec{v})}{2}^2} + \intOt{\lnorm{((\vec{v}_k - \vec{v}) \cdot \nabla)\vec{v}}{2}^2} \\
    &\qquad \lsim \intOt{\lnorm{\vec{v}_k}{\infty}^2 \lnorm{\nabla \vec{v}_k - \nabla \vec{v}}{2}^2} + \intOt{\lnorm{\nabla \vec{v}}{4}^2 \lnorm{\vec{v} - \vec{v}_k}{4}^2} \\
    &\qquad \lsim \intOt{\hnorm{\vec{v}_k - \vec{v}}{1}^2} \to 0.
\end{align*}
The limits \eqref{eqn:convergence-velocity-2}--\eqref{eqn:convergence-velocity-4} follow in the exact same manner, completing the proof.
\end{proof}

\begin{lemma} \label{lem:convergence-magnetic-field}
Assume that $k \to \infty$. The non-linear term in \eqref{eqn:galerkin-magnetic-field} converges to the limit
\begin{align}
    \curl(\vec{v}_k \times \vec{B}_k) &\to \curl(\vec{v} \times \vec{B}), \label{eqn:convergence-magnetic-field-1}
\end{align}
in the sense of $L^2(0,t; \LB^2(\Omega))$, for all $t \in [0,T^*]$.
\end{lemma}
\begin{proof}
Using \eqref{eqn:curl-cross-product}, the convergences \eqref{eqn:weak-convergence} and \eqref{eqn:strong-convergence},
\begin{align*}
    &\intOt{\lnorm{\curl(\vec{v}_k \times \vec{B}_k) - \curl(\vec{v} \times \vec{B})}{2}^2} \\
    &\qquad \lsim \intOt{\lnorm{\curl((\vec{v}_k  - \vec{v}) \times  \vec{B}_k)}{2}^2} + \intOt{\lnorm{\curl(\vec{v} \times ( \vec{B}_k -  \vec{B}))}{2}^2} \\
    &\qquad \lsim \intOt{\lnorm{\nabla \vec{v}_k - \nabla \vec{v}}{2}^2 \lnorm{\vec{B}_k}{2}^2} + \intOt{\lnorm{ \vec{v}_k - \vec{v}}{2}^2 \lnorm{\nabla \vec{B}_k}{2}^2}\\
    &\qquad \qquad + \intOt{\lnorm{\nabla \vec{v}}{2}^2 \lnorm{ \vec{B}_k -   \vec{B}}{2}^2} + \intOt{\lnorm{\vec{v}}{2}^2 \lnorm{\nabla \vec{B}_k -  \nabla \vec{B}}{2}^2}  \\
    &\qquad \lsim \intOt{\hnorm{\vec{v}_k - \vec{v}}{1}^2} + \intOt{\hnorm{\vec{B}_k - \vec{B}}{1}^2} \to 0.
\end{align*}
This completes the proof.
\end{proof}

\begin{lemma} \label{lem:convergence-magnetisation}
Assume that $k \to \infty$. The non-linear terms in \eqref{eqn:galerkin-magnetisation} converge to the limit
\begin{align}
    (\vec{v}_k \cdot \nabla)\vec{m}_k &\to (\vec{v} \cdot \nabla)\vec{m}, \label{eqn:convergence-magnetisation-1} \\[2ex]
    |\nabla \vec{m}_k|^2 \vec{m}_k &\to |\nabla \vec{m}|^2\vec{m}, \label{eqn:convergence-magnetisation-2} \\[2ex]
    \vec{m}_k \times \Delta \vec{m}_k &\to \vec{m} \times \Delta \vec{m}, \label{eqn:convergence-magnetisation-3} \\[2ex]
    \vec{m}_k \times \vec{B}_k &\to \vec{m} \times \vec{B}, \label{eqn:convergence-magnetisation-4} \\[2ex]
    \vec{m}_k \times (\vec{m}_k \times \vec{B}_k) &\to \vec{m} \times (\vec{m} \times \vec{B}), \label{eqn:convergence-magnetisation-5}
\end{align}
in the sense of $L^2(0,t; \LB^2(\Omega))$, for all $t \in [0,T^*]$.
\end{lemma}
\begin{proof}
The limit \eqref{eqn:convergence-magnetisation-1} follows in the same manner as \eqref{eqn:convergence-velocity-1} in Lemma \ref{lem:convergence-velocity}. For \eqref{eqn:convergence-magnetisation-2} we use the Sobolev embedding $\HB^1(\Omega) \hookrightarrow \LB^6(\Omega) \hookrightarrow \LB^4(\Omega)$ to obtain
\begin{align*}
    &\intOt{\lnorm{|\nabla \vec{m}_k|^2 \vec{m}_k - |\nabla \vec{m}|^2 \vec{m}}{2}^2} \\
    &\qquad \lsim \intOt{\lnorm{|\nabla \vec{m}_k|^2 (\vec{m}_k - \vec{m})}{2}^2} + \intOt{\lnorm{(|\nabla \vec{m}_k|^2 - |\nabla \vec{m}|^2) \vec{m}}{2}^2}\\
    &\qquad \lsim \intOt{\lnorm{\nabla \vec{m}_k}{4}^4 \lnorm{\vec{m}_k - \vec{m}}{2}^2} \\
    &\qquad \qquad + \intOt{\lnorm{\nabla \vec{m}_k - \nabla \vec{m}}{2}^2 \lnorm{\nabla \vec{m}_k + \nabla \vec{m}}{4}^2 \lnorm{\vec{m}}{4}^2} \\
    &\qquad \lsim \intOt{\hnorm{\vec{m}_k - \vec{m}}{2}^2} \to 0.
\end{align*}
The limits \eqref{eqn:convergence-magnetisation-3}--\eqref{eqn:convergence-magnetisation-5} follow in the same manner as \eqref{eqn:convergence-velocity-2}. This completes the proof. 
\end{proof}

The convergence of the non-linear terms are now proved in the context of the weak formulation as a result of the previous lemmas.

\begin{lemma} \label{lem:convergence-weak-formulation}
Let $\{(\vec{\phi}_k, \vec{\chi}_k, \vec{\psi}_k)\}$ be a sequence in $\V_k \times \XB_k \times \YB_k$ converging to $(\vec{\phi}, \vec{\chi}, \vec{\psi})$ in $\HB^2(\Omega)$. For all $t \in [0,T^*]$ the non-linear terms in the weak formulation for the velocity \eqref{eqn:galerkin-velocity} converge to
\begin{align}
    \lim_{k \to \infty} \intOt{\liprod{(\vec{v}_k \cdot \nabla) \vec{v}_k}{\boldsymbol{\phi}_k}} &= \intOt{\liprod{(\vec{v} \cdot \nabla)\vec{v}}{\boldsymbol{\phi}}}, \label{eqn:convergence-weak-formulation-velocity-1} \\
    \lim_{k \to \infty} \intOt{\liprod{\curl \vec{B}_k \times \vec{B}_k}{\boldsymbol{\phi}_k}} &= \intOt{\liprod{\curl \vec{B} \times \vec{B}}{\boldsymbol{\phi}}}, \label{eqn:convergence-weak-formulation-velocity-2} \\
    \lim_{k \to \infty} \intOt{\liprod{(\vec{B}_k \cdot \nabla) \vec{m}_k}{\boldsymbol{\phi}_k}} &= \intOt{\liprod{(\vec{B} \cdot \nabla)\vec{m}}{\boldsymbol{\phi}}}, \label{eqn:convergence-weak-formulation-velocity-3} \\
    \lim_{k \to \infty} \intOt{\liprod{(\nabla \vec{m}_k)^{\top} \Delta \vec{m}_k}{\vec{\phi}_k}} &= \intOt{\liprod{(\nabla \vec{m})^{\top} \Delta \vec{m}}{\vec{\phi}}}, \label{eqn:convergence-weak-formulation-velocity-4}
\end{align}
while the non-linear terms in the weak formulation for the magnetic field \eqref{eqn:galerkin-magnetic-field} converge to
\begin{align}
    \lim_{k \to \infty} \intOt{\liprod{\curl(\vec{v}_k \times \vec{B}_k)}{\boldsymbol{\chi}_k}} &= \intOt{\liprod{\curl(\vec{v} \times \vec{B})}{\boldsymbol{\chi}}}, \label{eqn:convergence-weak-formulation-magnetic-field-1}
\end{align}
and lastly the non-linear terms in the weak formulation for the magnetisation \eqref{eqn:galerkin-magnetisation} converge to
\begin{align}
    \lim_{k \to \infty} \intOt{\liprod{(\vec{v}_k \cdot \nabla) \vec{m}_k}{\boldsymbol{\psi}_k}} &= \intOt{\liprod{(\vec{v} \cdot \nabla)\vec{m}}{\boldsymbol{\psi}}}, \label{eqn:convergence-weak-formulation-magnetisation-1} \\
    \lim_{k \to \infty} \intOt{\liprod{|\nabla \vec{m}_k|^2 \vec{m}_k}{\boldsymbol{\psi}_k}} &= \intOt{\liprod{|\nabla \vec{m}|^2 \vec{m}}{\boldsymbol{\psi}}}, \label{eqn:convergence-weak-formulation-magnetisation-2} \\
    \lim_{k \to \infty} \intOt{\liprod{\vec{m}_k \times \nabla \vec{m}_k}{\nabla \boldsymbol{\psi}_k}} &= \intOt{\liprod{\vec{m} \times \nabla \vec{m}}{\nabla \boldsymbol{\psi}}}, \label{eqn:convergence-weak-formulation-magnetisation-3} \\
    \lim_{k \to \infty} \intOt{\liprod{\vec{m}_k \times \vec{B}_k}{\boldsymbol{\psi}_k}} &= \intOt{\liprod{\vec{m} \times \vec{B}}{\boldsymbol{\psi}}}, \label{eqn:convergence-weak-formulation-magnetisation-4} \\
    \lim_{k \to \infty} \intOt{\liprod{\vec{m}_k \times (\vec{m}_k \times \vec{B}_k)}{\boldsymbol{\psi}_k}} &= \intOt{\liprod{\vec{m} \times (\vec{m} \times \vec{B})}{\boldsymbol{\psi}}}. \label{eqn:convergence-weak-formulation-magnetisation-5}
\end{align}
\end{lemma}
\begin{proof}
The terms \eqref{eqn:convergence-weak-formulation-velocity-1}--\eqref{eqn:convergence-weak-formulation-velocity-4} follow from Lemma \ref{lem:convergence-velocity} and the convergence $\vec{\phi}_k \to \vec{\phi}$ in $\HB^2(\Omega)$. The terms \eqref{eqn:convergence-weak-formulation-magnetic-field-1} and \eqref{eqn:convergence-weak-formulation-magnetisation-1}--\eqref{eqn:convergence-weak-formulation-magnetisation-5} follow in a similar manner by Lemma \ref{lem:convergence-magnetic-field} and Lemma \ref{lem:convergence-magnetisation}, respectively.
\end{proof}

Having passed to the limit, we show that the limiting magnetisation term $\vec{m}$ conserves the value of $|\vec{m}|$, in agreement with Proposition \ref{prop:implicit-unit-magnitude}.

\begin{proposition} \label{prop:unit-magnitude}
The magnetisation $\vec{m}$ satisfies $|\vec{m}| = 1$ almost everywhere on $[0,T^*] \times \Omega$.
\end{proposition}
\begin{proof}
From Lemma \ref{lem:convergence-weak-formulation}, we noted that $(\vec{v}, \vec{B}, \vec{m})$ satisfies the PDEs in \eqref{eqn:fmhd-equiv} almost everywhere. Taking the dot product of
\begin{equation*}
    \frac{\partial \vec{m}}{\partial t} + (\vec{v} \cdot \nabla) \vec{m} = \Delta \vec{m} + |\nabla \vec{m}|^2 \vec{m} + \vec{m} \times (\Delta \vec{m} + \vec{B}) + \vec{m} \times (\vec{m} \times \vec{B})
\end{equation*}
with $\vec{m}$, we obtain
\begin{equation*}
    \frac{1}{2} \frac{\partial}{\partial t} |\vec{m}|^2 + \frac{1}{2}(\vec{v} \cdot \nabla) |\vec{m}|^2 = (\Delta \vec{m} \cdot \vec{m}) + |\nabla \vec{m}|^2 |\vec{m}|^2.
\end{equation*}
The identity \eqref{eqn:laplace-magnitude} implies that
\begin{equation*}
    \frac{\partial}{\partial t} |\vec{m}|^2 + (\vec{v} \cdot \nabla) |\vec{m}|^2 - \Delta |\vec{m}|^2 = 2|\nabla \vec{m}|^2( |\vec{m}|^2 - 1).
\end{equation*}
Setting $b = |\vec{m}|^2 - 1$ we see that $b(t,\vec{x})$ satisfies the PDE
\begin{equation*}
    \frac{\partial b}{\partial t} + (\vec{v} \cdot \nabla) b - \Delta b = 2|\nabla \vec{m}|^2 b,
\end{equation*}
subject to the Neumann boundary condition $\partial_{\vec{n}} b = 0$ and the initial condition $b(0,\cdot) = |\vec{m}_0|^2 - 1 = 0$. Taking the $L^2(\Omega)$ inner-product with $b$, noting $\div \vec{v} = 0$ and using the Sobolev embedding $\HB^2(\Omega) \hookrightarrow \LB^{\infty}(\Omega)$ we have
\begin{align*}
    \frac{1}{2} \ddt \lnorm{b}{2}^2 + \lnorm{\nabla b}{2}^2 &\lsim \lnorm{\nabla \vec{m}}{\infty}^2 \lnorm{b}{2}^2 \\
    &\lsim \hnorm{\vec{m}}{3}^2 \lnorm{b}{2}^2 \\
    &\lsim \lnorm{b}{2}^2,
\end{align*}
where the latter inequality comes from \eqref{eqn:weak-convergence}. Therefore, by the Gronwall Lemma we conclude that
\begin{equation*}
    \lnorm{b(t,\cdot)}{2}^2 \leq \lnorm{b(0,\cdot)}{2}^2 \exp\left(C T^*\right),
\end{equation*}
i.e., $\lnorm{b(t,\cdot)}{2} = 0$. This ends the proof.
\end{proof}

From Lemma \ref{lem:convergence-weak-formulation} and \eqref{eqn:weak-convergence} by passing to the limit in \eqref{eqn:galerkin-velocity} and using integration by parts, we obtain
\begin{align*}
    \liprod{\partial_t \vec{v} + (\vec{v} \cdot \nabla) \vec{v} - \Delta \vec{v} - \curl \vec{B} \times \vec{B} - (\vec{B} \cdot \nabla) \vec{m} + (\nabla \vec{m})^{\top} \Delta \vec{m}}{\vec{\varphi}} = 0,
\end{align*}
for all $\vec{\varphi} \in \HB_{0,\div}^1(\Omega)$. By using \cite[Proposition 1.1]{temam2001-book}, which states that a distribution $\vec{f}$ satisfying
\begin{equation*}
    \vec{f} = \nabla g, 
\end{equation*}
for some distribution $g$ if and only if
\begin{equation*}
    \liprod{\vec{f}}{\vec{\chi}} = 0, \text{ for all } \vec{\chi} \in \mathbb{C}_c^{\infty}(\Omega) \text{ such that } \div \vec{\chi} = 0,
\end{equation*}
we can recover pressure term and the additional gradient terms present in \eqref{eqn:fmhd-equiv}. We thus conclude that $(\vec{v}, \vec{B}, \vec{m})$ satisfy the PDEs in \eqref{eqn:fmhd-equiv} almost everywhere. We now use these facts to prove the next two results.

\begin{proposition} \label{prop:W26-estimate}
The limiting velocity and magnetic field $\vec{v}$ and $\vec{B}$ satisfy
\begin{equation*}
    \intOt{\left(\wnorm{\vec{v}(s)}{2}{6}^2 + \wnorm{\vec{B}(s)}{2}{6}^2\right)} \lsim 1
\end{equation*}
for any $t \in [0,T^*]$.
\end{proposition}
\begin{proof}
From \eqref{eqn:fmhd-equiv a}, velocity $\vec{v}$ satisfies
\begin{equation*}
    -\Delta \vec{v} + \nabla p = -\partial_t \vec{v} -(\vec{v} \cdot \nabla)\vec{v} + \curl \vec{B} \times \vec{B} + (\vec{B} \cdot \nabla)\vec{m} + \nabla \left(\vec{m} \cdot \vec{B} - \frac{1}{2} |\nabla \vec{m}|^2\right) - (\nabla \vec{m})^{\top} \Delta \vec{m}.
\end{equation*}
Setting $\nu = 1$, $g = 0$, $\phi = 0$, $m=0$ and $\alpha = 6$ in \cite[Proposition 2.2]{temam2001-book}, we obtain the regularity result
\begin{equation*}
    \wnorm{\vec{u}}{2}{6} \lsim \lnorm{-\Delta \vec{u} + \nabla p}{6}.
\end{equation*}
Now, setting $\vec{u} = \vec{v}$, using the Sobolev embeddings $\HB^1(\Omega) \hookrightarrow \LB^6(\Omega)$ and $\HB^2(\Omega) \hookrightarrow \LB^{\infty}(\Omega)$ and the regularity in \eqref{eqn:weak-convergence} we obtain
\begin{align*}
    \lnorm{\nabla^2 \vec{v}}{6} &\lsim \lnorm{\partial_t \vec{v}}{6} + \lnorm{\vec{v}}{\infty} \lnorm{\nabla \vec{v}}{6} + \lnorm{\vec{B}}{\infty} \lnorm{\nabla \vec{B}}{6} + \lnorm{\vec{B}}{\infty} \lnorm{\nabla \vec{m}}{6} \\
    &\qquad + \lnorm{\nabla \vec{m}}{\infty} \lnorm{\nabla^2 \vec{m}}{6} + \lnorm{\nabla \vec{m}_k}{\infty} \lnorm{\Delta \vec{m}_k}{6} \\
    &\lsim 1 + \hnorm{\partial_t \vec{v}}{1}. \\
    &\lsim 1 + \lnorm{\partial_t \nabla \vec{v}}{2}.
\end{align*}
Therefore, the Sobolev embedding $\W^{1,6}(\Omega) \hookrightarrow \HB^2(\Omega)$ and \eqref{eqn:weak-convergence} implies
\begin{align*}
    \intOt{\wnorm{\vec{v}}{2}{6}^2} &\lsim \intOt{\wnorm{\vec{v}}{1}{6}^2 + \lnorm{\nabla^2 \vec{v}}{6}^2} \\
    &\lsim \intOt{1 + \hnorm{\vec{v}}{2}^2 + \lnorm{\partial_s \nabla \vec{v}}{2}^2} \\
    &\lsim 1.
\end{align*}
For magnetic field $\vec{B}$, the identity \eqref{eqn:curl-squared} and \eqref{eqn:fmhd-equiv c} imply
\begin{equation*}
    -\Delta \vec{B} = \curl^2 \vec{B} = -\partial_t \vec{B} + \curl(\vec{v} \times \vec{B}).
\end{equation*}
Using the Sobolev embedding $\HB^1(\Omega) \hookrightarrow \LB^6(\Omega)$, we note that
\begin{equation*}
    \lnorm{\nabla^2 \vec{B}}{6} \lsim \hnorm{\vec{B}}{3}.
\end{equation*}
Hence, by setting $s=1$ in the regularity result \eqref{eqn:curl-curl-regularity}, and using \eqref{eqn:curl-cross-product} as in the proof of Lemma \ref{lem:curl-vxb}, we obtain
\begin{align*}
    \lnorm{\nabla^2 \vec{B}}{6} &\lsim \hnorm{-\partial_t \vec{B} + \curl(\vec{v} \times \vec{B})}{1} \\
    &\lsim \lnorm{\partial_t \vec{B}}{2} + \lnorm{\vec{B}}{\infty} \lnorm{\nabla \vec{v}}{2} + \lnorm{\vec{v}}{\infty} \lnorm{\nabla \vec{B}}{2} + \lnorm{\partial_t \nabla \vec{B}}{2} \\
    &\qquad + \lnorm{\nabla \vec{B}}{4} \lnorm{\nabla \vec{v}}{4} + \lnorm{\vec{v}}{\infty} \lnorm{\nabla^2 \vec{B}}{2} + \lnorm{\vec{B}}{\infty} \lnorm{\nabla^2 \vec{v}}{2} \\
    &\lsim 1 + \lnorm{\partial_t \nabla \vec{B}}{2}.
\end{align*}
The result for the magnetic field $\vec{B}$ is proved in the same way as the velocity $\vec{v}$, thus completing the proof. 
\end{proof}

Lastly, we prove that our limiting solutions satisfy the given initial data in the following proposition.

\begin{proposition} \label{prop:initial-data}
The limiting solutions $\vec{v}$, $\vec{B}$ and $\vec{m}$ obtained from \eqref{eqn:weak-convergence} and \eqref{eqn:strong-convergence} satisfies
\begin{equation*}
    \vec{v}(0) = \vec{v}_0, \qquad \vec{B}(0) = \vec{B}_0, \qquad \text{and} \qquad \vec{m}(0) = \vec{m}_0,
\end{equation*}
almost everywhere.
\end{proposition}
\begin{proof}
Integrating \eqref{eqn:galerkin-velocity} with respect to time over $[0,t]$, and using \eqref{eqn:weak-convergence}, Lemma \ref{lem:convergence-weak-formulation} and the fact that $\vec{v}_0^k \to \vec{v}_0$ in $\HB^2(\Omega)$ from \eqref{eqn:galerkin-initial-data}, we conclude that $\vec{v}$ satisfies \ref{eqn:weak-formulation-velocity}. However, differentiating \eqref{eqn:weak-formulation-velocity} with respect to time and integrating, once again, over $[0,t]$ we see that
\begin{align} \label{eqn:alternative-weak-formulation-velocity}
    \begin{split}
        &\liprod{\vec{v}(t)}{\boldsymbol{\varphi}}  + \intOt{\liprod{\nabla \vec{v}(s)}{\nabla \boldsymbol{\varphi}}}   +\intOt{\liprod{(\vec{v}(s) \cdot \nabla) \vec{v}(s)}{\boldsymbol{\varphi}}} - \intOt{\liprod{\curl \vec{B}(s) \times \vec{B}(s)}{\boldsymbol{\varphi}}}\\
        &\qquad - \intOt{\liprod{(\vec{B}(s) \cdot \nabla) \vec{m}(s)}{\boldsymbol{\varphi}}} + \intOt{\liprod{(\nabla \vec{m})^{\top} \Delta \vec{m}}{\vec{\varphi}}} = \liprod{\vec{v}(0)}{\boldsymbol{\varphi}},   
    \end{split}
\end{align}
where $\vec{v}(0)$ is well-defined since $\vec{v} \in C([0,T]; \HB^1(\Omega))$. Subtracting \eqref{eqn:weak-formulation-velocity} from \eqref{eqn:alternative-weak-formulation-velocity} we obtain
\begin{equation*}
    \liprod{\vec{v}(0) - \vec{v}_0}{\vec{\varphi}} = 0,
\end{equation*}
for all $\vec{\varphi} \in \HB^1(\Omega)$, and so, $\vec{v}(0) = \vec{v}_0$ almost everywhere. The proof is completed by repeating this process for $\vec{B}$ and $\vec{m}$. 
\end{proof}

Finally, we note that $\div \vec{v} = 0$ almost everywhere on $(0,T) \times \Omega$ since
\begin{align*}
    \intOt{\lnorm{\div \vec{v}}{2}^2} &\lsim \intOt{\lnorm{\div \vec{v}_k - \div \vec{v}}{2}^2} \\
    &\lsim \intOt{\hnorm{\vec{v}_k - \vec{v}}{1}^2} \to 0.
\end{align*}
Similarly, $\div \vec{B} = 0$. Having established the existence of bounded Galerkin solutions in Section \ref{sec:faedo-galerkin} we have passed them to the limit in \eqref{eqn:weak-convergence} and \eqref{eqn:strong-convergence}. The subsequent analysis of Lemma \ref{lem:convergence-weak-formulation} and Propositions \ref{prop:unit-magnitude}--\ref{prop:initial-data}, and recovery of gradient/pressure terms as well as the divergence free requirements of $\vec{v}$ and $\vec{B}$, we have uncovered strong solutions to \ref{eqn:fmhd-equiv} in the sense of Definition \ref{def:weak-solution}. This completes the proof of Theorem \ref{thm:existence}.

\section{Proof of Theorem \ref{thm:stability}} \label{sec:stability}

Let $(\vec{v}_1, \vec{B}_1, \vec{m}_1, p_1)$ and $(\vec{v}_2, \vec{B}_2, \vec{m}_2, p_2)$ be two solutions to \eqref{eqn:fmhd-equiv} from Theorem \ref{thm:existence} with initial data $(\vec{v}_{1,0}, \vec{B}_{1,0}, \vec{m}_{1,0})$ and $(\vec{v}_{2,0}, \vec{B}_{2,0}, \vec{m}_{2,0})$, respectively and additionally define
\begin{align*}
    \bar{p} &= p_1 - p_2.
\end{align*}
As $(\vec{v}_1, \vec{B}_1, \vec{m}_1, p_1)$ and $(\vec{v}_2, \vec{B}_2, \vec{m}_2, p_2)$ are strong solutions to \eqref{eqn:fmhd-equiv}, $(\bar{\vec{v}}, \bar{\vec{B}}, \bar{\vec{m}}, \bar{p})$ satisfies
\begin{align}
\begin{split} \label{eqn:diff-eqn-1}
    \partial_t \bar{\vec{v}} - \Delta \bar{\vec{v}} &= -(\vec{v}_1 \cdot \nabla) \bar{\vec{v}} - (\bar{\vec{v}} \cdot \nabla) \vec{v}_2 + \nabla \bar{p} + \curl \vec{B}_1 \times \bar{\vec{B}} + \curl \bar{\vec{B}} \times \vec{B}_2 \\
    &\qquad + (\vec{B}_1 \cdot \nabla) \bar{\vec{m}} + (\bar{\vec{B}} \cdot \nabla) \vec{m}_2 + \nabla (\vec{m}_1 \cdot \vec{B}_1 - \vec{m}_2 \cdot \vec{B}_2) \\
    &\qquad + (\nabla \vec{m}_1)^{\top} \Delta \bar{\vec{m}} + (\nabla \bar{\vec{m}})^{\top} \Delta \vec{m}_2,
\end{split} \\[2ex]
\begin{split} \label{eqn:diff-eqn-2}
    \partial_t \bar{\vec{B}} + \curl^2 \bar{\vec{B}} &= \curl(\vec{v}_1 \times \bar{\vec{B}}) + \curl(\bar{\vec{v}} \times \vec{B}_2),
\end{split} \\[2ex]
\begin{split} \label{eqn:diff-eqn-3}
    \partial_t \bar{\vec{m}} - \Delta \bar{\vec{m}} &= (\vec{v}_1 \cdot \nabla) \bar{\vec{m}} + (\bar{\vec{v}} \cdot \nabla) \vec{m}_2 + \vec{m}_1 \times \Delta \bar{\vec{m}} + \bar{\vec{m}} \times \Delta \vec{m}_2 \\
    &\qquad + |\nabla \vec{m}_1|^2 \bar{\vec{m}} + (|\nabla \vec{m}_1|^2 - |\nabla \vec{m}_2|^2) \vec{m}_2 \\
    &\qquad + \vec{m}_1 \times \bar{\vec{B}} + \bar{\vec{m}} \times \vec{B}_2 + \vec{m}_1 \times (\vec{m}_1 \times \bar{\vec{B}}) + \vec{m}_1 \times (\bar{\vec{m}} \times \vec{B}_2) + \bar{\vec{m}} \times (\vec{m}_2 \times \vec{B}_2),
\end{split}
\end{align}
almost everywhere on $[0,T^*] \times \Omega$. Note that we regularly make use of the regularity given by Theorem \ref{thm:existence}. Now, taking the $\LB^2(\Omega)$ inner product of \eqref{eqn:diff-eqn-1} with $\bar{\vec{v}}$, using the Sobolev embeddings $\HB^1(\Omega) \hookrightarrow \LB^6(\Omega) \hookrightarrow \LB^3(\Omega)$ and $\HB^2(\Omega) \hookrightarrow \LB^{\infty}(\Omega)$, noting that $\div \vec{v}_1 = 0$ and using Young's inequality for some $\varepsilon > 0$, we obtain
\begin{align*}
    &\frac{1}{2} \ddt \lnorm{\bar{\vec{v}}}{2}^2 + \lnorm{\nabla \bar{\vec{v}}}{2}^2 \\
    &\qquad \lsim \lnorm{\nabla \vec{v}_2}{\infty} \lnorm{\bar{\vec{v}}}{2}^2 + \lnorm{\curl \vec{B}_1}{\infty} \lnorm{\bar{\vec{B}}}{2} \lnorm{\bar{\vec{v}}}{2} \\
    &\qquad \qquad + \lnorm{\vec{B}_2}{\infty} \lnorm{\curl \bar{\vec{B}}}{2} \lnorm{\bar{\vec{v}}}{2} + \lnorm{\vec{B}_1}{\infty} \lnorm{\bar{\vec{m}}}{2} \lnorm{\bar{\vec{v}}}{2} \\
    &\qquad \qquad + \lnorm{\nabla \vec{m}_2}{\infty} \lnorm{\bar{\vec{B}}}{2} \lnorm{\bar{\vec{v}}}{2} + \lnorm{\nabla \vec{m}_1}{\infty} \lnorm{\Delta \bar{\vec{m}}}{2} \lnorm{\bar{\vec{v}}}{2} \\
    &\qquad \qquad + \lnorm{\bar{\vec{v}}}{6} \lnorm{\Delta \vec{m}_2}{3} \lnorm{\nabla \bar{\vec{m}}}{2} \\
    &\qquad \lsim \left(1 + \lnorm{\nabla \vec{v}_2}{\infty}\right) \lnorm{\bar{\vec{v}}}{2}^2 + \left(1 + \lnorm{\nabla \vec{B}_1}{\infty}^2\right) \lnorm{\bar{\vec{B}}}{2}^2 \\
    &\qquad \qquad + \lnorm{\bar{\vec{m}}}{2}^2 + \lnorm{\nabla \bar{\vec{m}}}{2}^2 + \varepsilon \lnorm{\nabla \bar{\vec{v}}}{2}^2 + \varepsilon \lnorm{\nabla \bar{\vec{B}}}{2}^2 + \varepsilon \lnorm{\Delta \bar{\vec{m}}}{2}^2.
\end{align*}
Similarly, taking the $\LB^2(\Omega)$ inner product of \eqref{eqn:diff-eqn-2} with $\bar{\vec{B}}$ and using Lemma \ref{lem:curl-vxb}, we obtain 
\begin{align*}
    \frac{1}{2} \ddt \lnorm{\bar{\vec{B}}}{2}^2 + \lnorm{\curl \bar{\vec{B}}}{2}^2 &\lsim \lnorm{|\bar{\vec{B}}| \, |\nabla \vec{v}_1| \, |\bar{\vec{B}}|}{1} + \lnorm{|\vec{v}_1| \, |\nabla \bar{\vec{B}}| \, |\bar{\vec{B}}|}{1} \\
    &\qquad + \lnorm{|\vec{B}_2| \, |\nabla \bar{\vec{v}}| \, |\bar{\vec{B}}|}{1} + \lnorm{|\bar{\vec{v}}| \, |\nabla \vec{B}_2| \, |\bar{\vec{B}}|}{1} \\
    &\lsim \lnorm{\nabla \vec{v}_1}{\infty} \lnorm{\bar{\vec{B}}}{2}^2 + \lnorm{\vec{v}_1}{\infty} \lnorm{\nabla \bar{\vec{B}}}{2} \lnorm{\bar{\vec{B}}}{2} \\
    &\qquad + \lnorm{\vec{B}_2}{\infty} \lnorm{\nabla \bar{\vec{v}}}{2} \lnorm{\bar{\vec{B}}}{2} \\
    &\qquad + \lnorm{\nabla \vec{B}_2}{\infty} \lnorm{\bar{\vec{v}}}{2} \lnorm{\bar{\vec{B}}}{2}, \\
    &\lsim \lnorm{\bar{\vec{v}}}{2}^2 + \left(1 + \lnorm{\nabla \vec{v}_1}{\infty} + \lnorm{\nabla \vec{B}_2}{\infty}^2\right) \lnorm{\bar{\vec{B}}}{2}^2 \\
    &\qquad + \varepsilon \lnorm{\nabla \bar{\vec{v}}}{2}^2 + \varepsilon \lnorm{\nabla \bar{\vec{B}}}{2}^2.
\end{align*}
The magnetisation equation \eqref{eqn:diff-eqn-3} requires two steps. But before that, we note that
\begin{align*}
    \left||\nabla \vec{m}_1|^2 - |\nabla \vec{m}_2|^2\right| \leq \left(|\nabla \vec{m}_1| + |\nabla \vec{m}_2|\right) |\nabla \bar{\vec{m}}|.
\end{align*}
Now we are in a position to estimate the difference. First, in a similar manner, we take the $\LB^2(\Omega)$ inner product of \eqref{eqn:diff-eqn-3} with $\bar{\vec{m}}$, note that $\div \vec{v}_1 = 0$ and apply integration by parts to the third right-hand side term to obtain
\begin{align*}
    \frac{1}{2} \ddt \lnorm{\bar{\vec{m}}}{2}^2 + \lnorm{\nabla \bar{\vec{m}}}{2}^2 &\lsim \lnorm{\nabla \vec{m}_2}{\infty} \lnorm{\bar{\vec{v}}}{2} \lnorm{\bar{\vec{m}}}{2} + \lnorm{\nabla \vec{m}_1}{\infty} \lnorm{\nabla \bar{\vec{m}}}{2} \lnorm{\bar{\vec{m}}}{2}  \\
    &\qquad + \lnorm{\nabla \vec{m}_1}{\infty}^2 \lnorm{\bar{\vec{m}}}{2}^2 \\
    &\qquad + \left(\lnorm{\nabla \vec{m}_1}{\infty} + \lnorm{\nabla \vec{m}_2}{\infty}\right) \lnorm{\vec{m}_2}{\infty} \lnorm{\nabla \bar{\vec{m}}}{2} \lnorm{\bar{\vec{m}}}{2} \\
    &\qquad + \lnorm{\vec{m}_1}{\infty} \lnorm{\bar{\vec{B}}}{2} \lnorm{\bar{\vec{m}}}{2} + \lnorm{\vec{m}_1}{\infty}^2 \lnorm{\bar{\vec{B}}}{2} \lnorm{\bar{\vec{m}}}{2} \\
    &\qquad + \lnorm{\vec{m}_1}{\infty} \lnorm{\vec{B}_2}{\infty} \lnorm{\bar{\vec{m}}}{2}^2 \\[1ex]
    &\lsim \lnorm{\bar{\vec{v}}}{2}^2 + \lnorm{\bar{\vec{B}}}{2}^2  + \lnorm{\bar{\vec{m}}}{2}^2 + \lnorm{\nabla \bar{\vec{m}}}{2}^2.
\end{align*}
Then we take the $\LB^2(\Omega)$ inner product of \eqref{eqn:diff-eqn-3} with $-\Delta \bar{\vec{m}}$ and integrate by parts so that we have
\begin{align*}
    &\frac{1}{2} \ddt \lnorm{\nabla \bar{\vec{m}}}{2}^2 + \lnorm{\Delta \bar{\vec{m}}}{2}^2 \\
    &\qquad \lsim \lnorm{\vec{v}_1}{\infty} \lnorm{\nabla \bar{\vec{m}}}{2} \lnorm{\Delta \bar{\vec{m}}}{2} + \lnorm{\nabla \vec{m}_2}{\infty} \lnorm{\bar{\vec{v}}}{2} \lnorm{\Delta \bar{\vec{m}}}{2} \\
    &\qquad \qquad + \lnorm{\nabla \vec{m}_1}{\infty}^2 \lnorm{\bar{\vec{m}}}{2} \lnorm{\Delta \bar{\vec{m}}}{2} \\
    &\qquad \qquad + \left(\lnorm{\nabla \vec{m}_1}{\infty} + \lnorm{\nabla \vec{m}_2}{\infty}\right) \lnorm{\vec{m}_2}{\infty} \lnorm{\nabla \bar{\vec{m}}}{2} \lnorm{\Delta \bar{\vec{m}}}{2} \\
    &\qquad \qquad + \lnorm{\vec{m}_1}{\infty} \lnorm{\bar{\vec{B}}}{2} \lnorm{\Delta \bar{\vec{m}}}{2} + \lnorm{\vec{B}_2}{\infty} \lnorm{\bar{\vec{m}}}{2} \lnorm{\Delta \bar{\vec{m}}}{2} \\
    &\qquad \qquad + \lnorm{\vec{m}_1}{\infty}^2 \lnorm{\bar{\vec{B}}}{2} \lnorm{\Delta \bar{\vec{m}}}{2} + \lnorm{\vec{m}_1}{\infty} \lnorm{\vec{B}_2}{\infty} \lnorm{\bar{\vec{m}}}{2} \lnorm{\Delta \bar{\vec{m}}}{2} \\
    &\qquad \qquad + \lnorm{\vec{m}_2}{\infty} \lnorm{\vec{B}_2}{\infty} \lnorm{\bar{\vec{m}}}{2} \lnorm{\Delta \bar{\vec{m}}}{2} \\
    &\qquad \lsim \lnorm{\bar{\vec{v}}}{2}^2 + \lnorm{\bar{\vec{B}}}{2}^2 + \lnorm{\bar{\vec{m}}}{2}^2 + \lnorm{\nabla \bar{\vec{m}}}{2}^2 + \varepsilon \lnorm{\Delta \bar{\vec{m}}}{2}^2.
\end{align*}
Using Lemma \ref{lem:curl-estimate} and assuming that $\varepsilon > 0$ is sufficiently small, the sum of these estimates satisfy
\begin{align*}
    &\frac{1}{2} \ddt \left(\lnorm{\bar{\vec{v}}}{2}^2 + \lnorm{\bar{\vec{B}}}{2}^2 + \lnorm{\bar{\vec{m}}}{2}^2 + \lnorm{\nabla \bar{\vec{m}}}{2}^2\right) \\
    &\qquad + \left(\lnorm{\nabla \bar{\vec{v}}}{2}^2 + \lnorm{\nabla \bar{\vec{B}}}{2}^2 + \lnorm{\nabla \bar{\vec{m}}}{2}^2 + \lnorm{\Delta \bar{\vec{m}}}{2}^2\right) \\
    &\qquad \qquad \lsim \left(1 + \lnorm{\nabla \vec{v}_2}{\infty}\right) \lnorm{\bar{\vec{v}}}{2}^2 + \left(1 + \lnorm{\nabla \vec{v}_1}{\infty} + \lnorm{\nabla \vec{B}_2}{\infty}^2\right) \lnorm{\bar{\vec{B}}}{2}^2 \\
    &\qquad \qquad \qquad +\lnorm{\bar{\vec{m}}}{2}^2 + \lnorm{\nabla \bar{\vec{m}}}{2}^2.
\end{align*}
By Young's inequality, we note that
\begin{equation*}
    1 + \lnorm{\nabla \vec{v}_1}{\infty} \lsim 1 + \lnorm{\nabla \vec{v}_1}{\infty}^2,
\end{equation*}
and similarly for $\vec{v}_2$. Therefore, using the Sobolev embedding $\W^{1,6}(\Omega) \hookrightarrow \LB^{\infty}(\Omega)$ and the Gronwall Lemma we obtain
\begin{align*}
    &\left(\lnorm{\bar{\vec{v}}}{2}^2 + \lnorm{\bar{\vec{B}}}{2}^2 + \hnorm{\bar{\vec{m}}}{1}^2\right) \\
    &\qquad \lsim \left(\lnorm{\bar{\vec{v}}_0}{2}^2 + \lnorm{\bar{\vec{B}}_0}{2}^2 + \hnorm{\bar{\vec{m}}_0}{1}^2\right) \\
    &\qquad \qquad \exp \left(\int_0^{T^*} 1 + \wnorm{\vec{v}_1}{2}{6}^2 + \wnorm{\vec{v}_2}{2}{6}^2  + \wnorm{\vec{B}_2}{2}{6}^2 \; {\rm d}s\right).
\end{align*}
The required result is proved by applying Proposition \ref{prop:W26-estimate}. This completes the proof of Theorem \ref{thm:stability}.

\section*{Acknowledgements}
The first author is supported by the Australian Government's Research Training Program Scholarship awarded at the University of New South Wales, Sydney. The second author is partially supported by the Australian Research Council under grant number DP190101197 and DP200101866.

\appendix
\section{Useful Results} \label{app:useful-results}

\begin{theorem}[Aubin's Lemma \cite{aubin1963-article, simon1986-article}] \label{aubins-lemma}
Let $X$, $Y$ and $B$ be Banach spaces such that $X \subset B \subset Y$, where the injection $X \subset B$ is compact and the injection $B \subset Y$ is continuous. Assume that $\{u_k\}_{k=1}^{\infty}$ is a bounded sequence in $L^p(0,T; X)$ such that $\{\partial_t u_k\}_{k=1}^{\infty}$ is bounded in $L^r(0,T; Y)$ where $1 \leq p < \infty$ and $r = 1$, or $p = \infty$ and $r > 1$. Then there exists a subsequence $\{u_{k_j}\}_{j=1}^{\infty}$ which strongly converges in $L^p(0,T; B)$. 
\end{theorem}

\begin{theorem}[Generalised Gronwall Lemma \cite{beckenbach1961-book}] \label{generalised-gronwall-lemma}
Assume $u : [a,b] \to [0,\infty)$ and $\beta : [a,b] \to [0,\infty)$ along with non-decreasing $g : [0,\infty) \to [0,\infty)$ satisfy
\begin{equation*}
    u(t) \leq \alpha + \int_a^t \beta(s) g(u(s)) \; {\rm d}s \quad \forall t \in [a,b],
\end{equation*}
where $\alpha$ is a positive constant and $[a,b] \subset [0,\infty)$. Then
\begin{equation*}
    u(t) \leq G^{-1}\left(\int_a^t \beta(s) \; {\rm d}s\right), \quad t \in \RR,
\end{equation*}
where $G^{-1}$ is the inverse function of
\begin{equation*}
    G(\sigma) := \int_{\alpha}^{\sigma} \frac{1}{g(s)} \; {\rm d}s, \quad \sigma \geq 0,
\end{equation*}
and
\begin{equation*}
    \RR = \left\{t \in [a,b] : \int_a^t \beta(s) \; {\rm d}s \in G([0,\infty))\right\}.
\end{equation*}
\end{theorem}

\begin{theorem}[Gagliardo-Nirenberg inequalities]\label{thm:gagliardo-nirenberg}
Let $\Omega$ be a bounded domain of $\R^d$ with Lipschitz boundary, and
let~$\vec{u}: \Omega \to \R^n$. Then
\begin{align}\label{gagliardo}
	\|\vec{u}\|_{\W^{r,q}(\Omega)} 
	\leq 
	C \|\vec{u}\|_{\HB^{s_1}(\Omega)}^\theta \|\vec{u}\|_{\HB^{s_2}(\Omega)}^{1-\theta}
\end{align}
for all $\vec{u} \in \HB^{s_2}(\Omega)$, where $s_1,s_2,r$ are non-negative
real numbers satisfying
\begin{align*}
	0\leq s_1<s_2,\quad \theta \in (0,1),\quad 0\leq r <\theta s_1 +(1-\theta)s_2, 
\end{align*}
and $q\in(2,\infty]$ satisfies
\[
	\frac{1}{q}=\frac{1}{2}+\frac{(s_2-s_1)\theta}{d}- \frac{s_2-r}{d}.
\]
Moreover, when $2< q<\infty$, we have
\begin{align*}
	\theta =\frac{2q(s_2-r)-d(q-2)}{2q(s_2-s_1)}.
%\quad \text{and} \quad \frac{2q(s_1-r)}{q-2} < d< \frac{2q(s_2-r)}{q-2}.
\end{align*}
\end{theorem}

\bibliographystyle{abbrv}
\bibliography{Bibliography}

\end{document}